\documentclass[11pt, leqno]{amsart}

\usepackage[svgnames, dvipsnames, table]{xcolor}
\usepackage[colorlinks = true, allcolors = Crimson, pagebackref=true, colorlinks]{hyperref}

\usepackage[utf8]{inputenc}
\usepackage{amssymb}
\usepackage{amscd}        
\usepackage{caption}
\usepackage{subcaption}
\usepackage{float}
\usepackage{longtable}
\usepackage{booktabs}
\usepackage{multirow}
\usepackage{lipsum}
\usepackage{siunitx}
\usepackage{nomencl}
\usepackage{setspace}
\usepackage[ruled, linesnumbered, algosection]{algorithm2e}
\usepackage{pifont}
\usepackage{mathrsfs}
\usepackage{stmaryrd}
\usepackage{bm}
\usepackage{tikz-cd}
\usepackage{tikz}
\usepackage{pgfplots}
\pgfplotsset{compat=1.16}
\usepackage{verbatim}
\usepackage{xcolor}
\usepackage{colortbl}
\usepackage{booktabs}
\usepackage{tabu}
\usepackage[all]{xy}

\usepackage{listings}
\definecolor{sagebrown}{RGB}{176, 92, 10}
\definecolor{sageblue}{RGB}{44, 45, 254}
\definecolor{sagepurple}{RGB}{151, 57, 164}
\definecolor{sagegreen}{RGB}{18, 103, 68}
\definecolor{sagered}{RGB}{170, 16, 15}

\lstdefinestyle{SageMath-style}{
    backgroundcolor=\color{white},   
    commentstyle=\color{sagebrown},
    keywordstyle=\color{sagepurple},
    keywordstyle = [2]{\color{sageblue}},
    keywordstyle = [3]{\color{yellow}},
    numberstyle=\tiny\color{sagegreen},
    stringstyle=\color{sagered},
    basicstyle=\ttfamily\footnotesize,
    breakatwhitespace=false,         
    breaklines=true,                 
    captionpos=b,                    
    keepspaces=true,                 
    numbers=left,                    
    numbersep=5pt,                  
    showspaces=false,                
    showstringspaces=false,
    showtabs=false,                  
    tabsize=2
}

\usepackage{color}
\usepackage{graphicx}
\usepackage{rotating}
\usepackage{diagbox}
\usepackage{amssymb}
\usepackage{epstopdf}
\usepackage{tikz}
\definecolor{UCRceleste}{RGB}{0,192,243}
\definecolor{Crimson}{RGB}{220, 20, 60}
\definecolor{GaloisBlue}{HTML}{6495ED}
\definecolor{GaloisRed}{HTML}{DE3163}
\definecolor{mintgreen}{RGB}{152,255,152}
\definecolor{pinksalmon}{RGB}{255,102,102}
\definecolor{hueso}{RGB}{245,245,220}
\definecolor{marfil}{RGB}{255,253,208}
\definecolor{amarillo}{RGB}{255,255,0}
\usetikzlibrary{decorations.markings,arrows}
\usetikzlibrary{decorations.pathreplacing}
\usepackage[inner=1.0in,outer=1.0in,bottom=1.0in, top=1.0in]{geometry}

\newcommand{\hooklongrightarrow}{\lhook\joinrel\longrightarrow}

\numberwithin{equation}{section}

\newtheorem{theorem}{Theorem}[section]
\newtheorem{lemma}[theorem]{Lemma}
\newtheorem{proposition}[theorem]{Proposition}
\newtheorem{corollary}[theorem]{Corollary}
\newtheorem{conjecture}[theorem]{Conjecture}

\makeatletter
\def\moverlay{\mathpalette\mov@rlay}
\def\mov@rlay#1#2{\leavevmode\vtop{%
   \baselineskip\z@skip \lineskiplimit-\maxdimen
   \ialign{\hfil$\m@th#1##$\hfil\cr#2\crcr}}}
\newcommand{\charfusion}[3][\mathord]{
    #1{\ifx#1\mathop\vphantom{#2}\fi
        \mathpalette\mov@rlay{#2\cr#3}
      }
    \ifx#1\mathop\expandafter\displaylimits\fi}
\makeatother

\newcommand{\suchthat}{\;\ifnum\currentgrouptype=16 \middle\fi|\;}

\newcommand{\Z}{\mathbb{Z}}
\newcommand{\C}{\mathbb{C}}
\newcommand{\Q}{\mathbb{Q}}

\newcommand{\Gal}[1]{\operatorname{Gal}#1}

\newcommand{\op}[1]{\operatorname{#1}}

\newcommand{\CM}{\mathcal{CM}}
\newcommand{\sN}{\mathcal{N}}

\theoremstyle{definition}
\newtheorem{remark}[theorem]{Remark}
\newtheorem{definition}[theorem]{Definition}
\newtheorem{example}[theorem]{Example}

\usepackage{orcidlink}

\def\O_K{{\Cal{O}_{K}}}
\def\O_F{{\Cal{O}_{F}}}
\def\N_F{{\Cal{N}_{F/\Q}}}

\begin{document}

\title{On the embedding of Galois groups into wreath products}

\author[]{Adrian Barquero-Sanchez\orcidlink{0000-0001-7847-2938} and Jimmy Calvo-Monge\orcidlink{0000-0002-4823-2455}}


\address{Escuela de Matem\'atica, Universidad de Costa Rica, San Jos\'e 11501, Costa Rica}
\email{adrian.barquero\_s@ucr.ac.cr}
\email{jimmy.calvo@ucr.ac.cr}


\begin{abstract}
In this paper we make explicit an application of the wreath product construction to the Galois groups of field extensions. More precisely, given a tower of fields $F \subseteq K \subseteq L$ with $L/F$ finite and separable, we explicitly construct an embedding of the Galois group $\Gal(L^c/F)$ into the regular wreath product $\Gal(L^c/K^c) \wr_r \Gal(K^c/F)$. Here $L^c$ (resp. $K^c$) denotes the Galois closure of $L/F$ (resp. $K/F$). Similarly, we also construct an explicit embedding of the Galois group $\Gal(L^c/F)$ into the smaller sized wreath product $\Gal(L^c/K) \wr_{\Omega} \Gal(K^c/F)$, where $\Omega = \operatorname{Hom}_F(K, K^c)$ is acted on by composition of automorphisms in $\Gal(K^c/F)$. Moreover, when $L/K$ is a Kummer extension we prove a sharper embedding, that is, that $\Gal(L^c/F)$ embeds into the wreath product $\Gal(L/K) \wr_{\Omega} \Gal(K^c/F)$. As corollaries we obtain embedding theorems when $L/K$ is cyclic and when it is quadratic with $\operatorname{char}(F) \neq 2$. We also provide examples of these embeddings and as an illustration of the usefulness of these embedding theorems, we survey some recent applications of these types of results in field theory, arithmetic statistics, number theory and arithmetic geometry.
\end{abstract}

\maketitle

\section{Introduction}

Wreath products of groups constitute an essential construction in group theory and in different applications outside of pure group theory, with recent applications even in areas as unexpected as music theory! (see e.g. \cite{Hoo02}, \cite{Pec09} and the very nice blog post \cite{yistvanblog}). These groups appeared first in the context of permutation groups during the first half of the 20th century, where now they are of fundamental importance (see e.g. \cite{DM96}). They also appear naturally for example in the study of the extension problem for groups, where, along with direct products, they are the most common examples of group extensions. 

In Galois Theory, wreath products have also appeared in different contexts. One instance where they are used prominently is in the problem of realizing certain groups as Galois groups in connection with the inverse Galois problem (see e.g. \cite[Chapter 8]{Vol96}, \cite[Section IV.2]{MM18} and \cite{Gow86}). They have also appeared recently in the study of certain problems in enumerative geometry (see e.g. \cite{EL22}).

Nevertheless, as J. D. P. Meldrum wrote in the preface of his book \cite{Mel95}, \textit{``...in almost all cases where wreath products are used, in book or paper, the author has to develop the theory himself, using material from a variety of sources.''}. This was written in 1995, but it still remains partially valid almost thirty years later, for even though there are now several specialized books on group theory that treat wreath products to a certain extent, these groups are rarely even mentioned in the usual mathematics curriculum, even at the graduate level.

In particular, in this paper we prove several theorems giving different \textit{explicit} embeddings of Galois groups of separable extensions of fields into wreath products of certain associated field extensions. This type of results have important applications in different areas of mathematics. For example, in arithmetic statistics they can be used to study the distribution of number fields in the context of the Malle conjectures (see e.g. \cite{Mal02}, \cite{Klu12} and \cite{BSMT17}); in arithmetic geometry they have been used in defining the notion of \textit{Weyl CM points} on the moduli space $\mathcal{A}_{g, 1}$ of principally polarized abelian varieties (see e.g. \cite{CO12}) and in studying the Colmez conjecture from a probabilistic perspective, which even yielded a conditional proof that the conjecture is true with probability 1 (see e.g. \cite[\S 1.2]{BSM18} and \cite{BSMT17}); in field theory they have been used in characterizing the solvability by radicals of imprimitive polynomials (see e.g. \cite[\S 14.2]{Cox12}). These applications are surveyed in more detail in Section \ref{Applications} of the paper.

Some of the results of this paper are known to some experts in certain areas, possibly in a somewhat different form than the one given here. Some have even appeared in special cases in papers, books or in expository notes, but unfortunately there hasn't been a more complete and unified treatment of these embedding results as far as we know. Thus, we have strived to give general results under a unified approach, making clear use of the different versions of the wreath products involved and highlighting important similarities and differences between the results. We hope that this will help remedy the situation and contribute as a more standard reference, since, as was mentioned in the previous paragraph, this type of results have found important uses in different areas.

\subsection{Statement of main results}

We now proceed to state the main theorems of the paper. The first two theorems (Theorems \ref{Main-Theorem} and \ref{Main-Theorem-2}) are of a general nature and apply to towers of fields $F \subseteq K \subseteq L$ with $L/F$ finite and separable. Then, in Theorem \ref{Kummer-embedding} and its corollaries (Corollaries \ref{Cyclic-embedding} and \ref{Quadratic-embedding}), we put some conditions on the top extension $L/K$ in the tower, which allow us to get sharper embeddings, as will be explained below. The basic background and terminology on wreath products, as it is used in the statements of the theorems below, is given in Section \ref{Wreath-section} of the paper. Thus, our first theorem is the following. It gives an embedding of a Galois group into a regular wreath product, which, according to \cite[p. 5]{Mel95}, is the most commonly used wreath product.

\begin{theorem}\label{Main-Theorem}
	Let $F \subseteq K \subseteq L$ be a tower of field extensions with $L/F$ finite and separable. Denote the Galois closures of $K/F$ and $L/F$ by $K^c$ and $L^c$, respectively. Then the Galois group $\Gal(L^c/F)$ embeds into the regular wreath product $\Gal(L^c/K^c) \wr_r \Gal(K^c/F)$. More precisely, consider the short exact sequence from Galois theory
	$$
		1 \longrightarrow \Gal(L^c/K^c) \overset{\iota}{\longrightarrow} \Gal(L^c/F) \overset{\varepsilon}{\longrightarrow} \Gal(K^c/F) \longrightarrow 1,
	$$
	where $\iota$ is the inclusion and $\varepsilon$ is the restriction map, i.e., if $\rho \in \Gal(L^c/F)$, then $\varepsilon(\rho) = \rho|_{K^c}$.
	Let $s \colon \Gal(K^c/F) \longrightarrow \Gal(L^c/F)$ be a right inverse of $\varepsilon$, i.e., suppose that $\varepsilon \circ s = \op{id}_{\Gal(K^c/F)}$. Consider the map 
	\begin{align}\label{Galois-embedding-1}
		\varphi \colon &\Gal(L^c/F) \hooklongrightarrow \Gal(L^c/K^c) \wr_r \Gal(K^c/F) \\
		&\quad \quad \qquad  \rho \longmapsto (\sigma_{\rho}, \varepsilon(\rho)) \notag
	\end{align}
	with $\sigma_{\rho} \colon \Gal(K^c/F) \longrightarrow \Gal(L^c/K^c)$ defined by 
	$$
	\sigma_{\rho}(\tau) := s(\tau)^{-1} \circ \rho \circ s \big( \varepsilon(\rho)^{-1} \circ \tau \big)
	$$ 
	for every $\tau \in \Gal(K^c/F)$. Then the map $\varphi$ is an injective group homomorphism.
\end{theorem}

\begin{remark}
As will be explained in Section \ref{Kaloujnine-Krasner-section}, this is an application of a very important embedding theorem of Kaloujnine and Krasner \cite{KK51b} from 1951.
\end{remark}

\begin{remark}
	Observe that the embedding (\ref{Galois-embedding-1}) is essentially breaking up an automorphism $\rho \in \Gal(L^c/F)$ into two pieces of information with respect to the tower of fields $F \subseteq K^c \subseteq L^c$, as Figure \ref{Embedding-diagram} shows. More precisely, the restriction $\varepsilon(\rho)$ of $\rho$ to $K^c$ gives an automorphism in the Galois group $\Gal(K^c/F)$, and thus $\varepsilon(\rho)$ retains all the information about $\rho$ relative to the bottom extension $K^c/F$ in Figure \ref{Embedding-diagram}. On the other hand, the function $\sigma_{\rho} \colon \Gal(K^c/F) \longrightarrow \Gal(L^c/K^c)$ retains all the information about $\rho$ relative to the top extension $L^c/K^c$.
	
	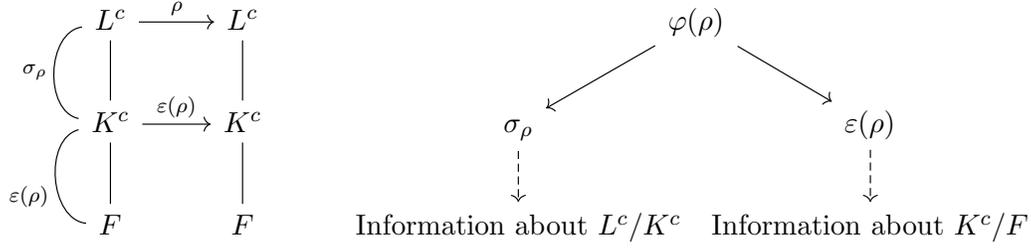
\begin{figure}[H]
		\begin{tikzcd}[/tikz/column 3/.style={column sep=-1.3em}, /tikz/column 4/.style={column sep=-1.3em}]
			L^c \arrow[r, "\rho"] \arrow[d, dash] \arrow[d, bend right = 80, dash, swap, "\sigma_{\rho}"] & L^c \arrow[d, dash] &  & \varphi(\rho) \arrow[dl] \arrow[dr] &  \\
			K^c \arrow[r, "\varepsilon(\rho)"] \arrow[d, dash] \arrow[d, bend right = 80, dash, swap, "\varepsilon(\rho)"] & K^c \arrow[d, dash] & \sigma_{\rho}  \arrow[d, dashed] &  & \varepsilon(\rho) \arrow[d, dashed]    \\
			F  & F & \text{Information about $L^c/K^c$} & & \text{Information about $K^c/F$}
		\end{tikzcd}
		\caption{How the embedding (\ref{Galois-embedding-1}) breaks up an automorphism.}
		\label{Embedding-diagram}
	\end{figure}
	
\end{remark}

A different embedding is given in the following theorem. In particular, this result gives an embedding into a wreath product which in general will have much smaller size than the regular wreath product appearing in Theorem \ref{Main-Theorem}, as will be clarified in Remarks \ref{Same-group-remark} and \ref{Size-Remark} below.

\begin{theorem}\label{Main-Theorem-2} 
    Consider a tower of fields $F \subseteq K \subseteq L$ with $L/F$ finite and separable. Let $L^c$ be the Galois closure of $L/F$ and $K^c$ denote the Galois closure of $K/F$. Let $\Omega:= \op{Hom}_F(K, K^c)$, and let $s: \Omega \longrightarrow \Gal(L^c/F)$ be a map that extends every $F$-embedding $\omega \in \Omega$ to an automorphism $s(\omega) \in \Gal(L^c/F)$. Define the map
    \begin{align*}
		\varphi \colon &\Gal(L^c/F) \hooklongrightarrow \Gal(L^c/K) \wr_{\Omega} \Gal(K^c/F) \\
		&\quad \quad \qquad  \rho \longmapsto (\sigma_{\rho}, \varepsilon(\rho)), \notag
	\end{align*}
    where $\sigma_{\rho}: \Omega = \op{Hom}_F(K, K^c) \longrightarrow \Gal(L^c/K)$ is given by 
    \begin{align*}
    \sigma_\rho(\omega) = s(\omega)^{-1} \circ \rho \circ s(\varepsilon(\rho)^{-1}\circ \omega)
    \end{align*}
    and $\varepsilon$ is the the restriction from  $\Gal(L^c/F)$ to $\Gal(K^c/F)$. Then, the map $\varphi$ is an injective group homomorphism, embedding the group $\Gal(L^c/F)$ into the wreath product $\Gal(L^c/K) \wr_{\Omega} \Gal(K^c/F)$.
\end{theorem}

\begin{remark}
We note that Theorem \ref{Main-Theorem-2} gives an explicit version of a result that was sketched by B. de Smit in an unpublished note \cite{BSmit07}, which also appears in the master's thesis of his student M. Pintonello \cite[Theorem 1.1.2]{Pinto18}. In both places, the result is obtained in the language of permutation groups and no explicit embedding is given. Here, we have avoided the more cumbersome language of permutation groups and moreover we give a direct and simpler proof of the result from the explicit form of the embedding that we obtained.
\end{remark}

\begin{remark}\label{Same-group-remark}
Observe that when $K/F$ is a Galois extension we have $K^c = K$. In particular, this implies that $\operatorname{Hom}_F(K, K^c) = \Gal(K^c/F)$ and $\Gal(L^c/K) = \Gal(L^c/K^c)$. Thus, this means that in this situation the two wreath products appearing in Theorems \ref{Main-Theorem} and \ref{Main-Theorem-2} are actually the same group and furthermore the embeddings coincide.
\end{remark}

\begin{remark}\label{Size-Remark}
Let $F \subseteq K \subseteq L$ be a tower of fields with $L/F$ finite and separable, as in the statement of Theorems \ref{Main-Theorem} and \ref{Main-Theorem-2}. Both  of those two theorems offer embeddings of $\Gal(L^c/F)$ into wreath products of associated Galois groups. Nevertheless, there is an important difference in the corresponding sizes of the wreath products involved. To explain this, referring ahead to Remark \ref{Wreath-size}, we know that the size of a wreath product $K \wr_{\Omega} H$ is given by $|K \wr_{\Omega} H| = |K|^{|\Omega|} \cdot |H|$. Hence, for the wreath product in Theorem \ref{Main-Theorem} we have
\begin{align}\label{Size-formula-regular-wreath}
|\Gal(L^c/K^c) \wr_r \Gal(K^c/F)| &= |\Gal(L^c/K^c)|^{|\Gal(K^c/F)|} |\Gal(K^c/F)| \\
&= \left( \frac{[L^c:F]}{[K^c:F]} \right)^{[K^c:F]} [K^c:F] \notag
\end{align}
and similarly for the wreath product appearing in Theorem \ref{Main-Theorem-2} we have
\begin{align}\label{Size-formula-omega-wreath}
|\Gal(L^c/K) \wr_{\Omega} \Gal(K^c/F)| &= |\Gal(L^c/K)|^{|\operatorname{Hom}_F(K, K^c)|} |\Gal(K^c/F)| \\
&= \left( \frac{[L^c:F]}{[K:F]} \right)^{[K:F]} [K^c:F]. \notag
\end{align}

As these two formulas show, the two wreath products have the same size (and in fact are actually the same group as was observed in Remark \ref{Same-group-remark}) when $K^c = K$, i.e., when $K/F$ is a Galois extension. Moreover, observe that by the tower law for field extensions we always have that $[K^c : F]$ divides $[L^c : F]$. 

Now, in Table \ref{Wreath-size-table} in Subsection \ref{tables-and-graphs-intro} we use the identities (\ref{Size-formula-regular-wreath}) and (\ref{Size-formula-omega-wreath}) to list formulas for the corresponding sizes for the different possibilities for $\Gal(K^c/F)$ when $2 \leq [K:F] \leq 5$. The formulas are given as functions of $m := |\Gal(L^c/F)| = [L^c:F]$. Moreover, since $[K^c:F]$ divides $[L^c:F]$, each fraction appearing in the table is actually an integer. In particular, the reader should note that the exponent of $m$ is in general larger for the regular wreath product $|\Gal(L^c/K^c) \wr_r \Gal(K^c/F)|$ than for the wreath product $\Gal(L^c/K) \wr_{\Omega} \Gal(K^c/F)$. This implies that, as functions of $m$, as $m$ gets bigger, the size of $\Gal(L^c/K^c) \wr_r \Gal(K^c/F)$ will be a lot larger than the size of $\Gal(L^c/K) \wr_{\Omega} \Gal(K^c/F)$. Therefore, in this respect, Theorem \ref{Main-Theorem-2} can be seen as giving a sharper embedding than Theorem \ref{Main-Theorem}, since one will in general be embedding the Galois group $\Gal(L^c/F)$ into a wreath product of much smaller size in Theorem \ref{Main-Theorem-2}.

Moreover, experimentally we have observed that when $K/F$ is non-Galois and when $m = [L^c : F] \geq 3[K^c : F]$, the regular wreath product $\Gal(L^c/K^c) \wr_r \Gal(K^c/F)$ is actually strictly bigger than the wreath product $\Gal(L^c/K) \wr_{\Omega} \Gal(K^c/F)$, and when $m = 2[K^c : F]$ they have the same size if and only if $\Gal(K^c/F) \simeq D_n$ is dihedral. This can be seen in Figures \ref{plot_K_F_3}, \ref{plot_K_F_4} and \ref{plot_K_F_5} at the end of Subsection \ref{tables-and-graphs-intro}.

\end{remark}

\subsection{Sharper embedding theorems for Kummer, cyclic and quadratic extensions}\label{Sharper-embeddings-subsection}

It turns out that one can get a sharper embedding when in the tower of fields $F \subseteq K \subseteq L$ we have $L/K$ quadratic. For example, C-L. Chai and F. Oort proved the following (see \cite[Lemma 2.8]{CO12}).

\begin{theorem}\label{Chai-Oort-Theorem}
Let $K$ be a number field of degree $n = [K : \Q]$ and let $L$ be a quadratic extension of $K$. Denote by $L^c$ the Galois closure of $L/\Q$. Then the Galois group $\Gal(L^c/\Q)$ is isomorphic to a subgroup of the wreath product $C_2 \wr_{\Omega} S_n$, where $\Omega = \{ 1, 2, \dots, n \}$ and the action of $S_n$ on $C_2^{\Omega} = C_2^n$ is by permuting the components as in Remark \ref{tuple-remark}. 
\end{theorem}

Moreover, under the same hypotheses as in the previous theorem, one can get an even sharper embedding. This was done by the first author, R. Masri and F. Thorne in \cite[Proposition 3.1]{BSMT17}.

\begin{theorem}\label{Barquero-Sanchez-Masri-Thorne-Theorem}
Let $K$ be a number field of degree $n = [K : \Q]$ and let $L$ be a quadratic extension of $K$. Denote by $L^c$ the Galois closure of $L/\Q$ and let $G$ be a transitive subgroup of $S_n$ with $\Gal(K^c/\Q) \simeq G$. Then $\Gal(L^c/\Q)$ is isomorphic to a subgroup of the wreath product $C_2 \wr_{\Omega} G$, where $\Omega = \{ 1, 2, \dots, n \}$ and the action of $G$ on $C_2^{\Omega} = C_2^n$ is by permuting the components as in Remark \ref{tuple-remark}.   
\end{theorem}

We will now give a very broad generalization of Theorems \ref{Chai-Oort-Theorem} and \ref{Barquero-Sanchez-Masri-Thorne-Theorem}. More precisely, we will prove an embedding theorem for towers of fields $F \subseteq K \subseteq L$ with $L/K$ a Kummer extension. We describe this now.

First, recall that if $K$ is a field, an element $\omega \in K$ with $\omega^n = 1$ is called an \textit{$n$-th root of unity}. Moreover, if the order of $\omega$ in the multiplicative group $K^{\times}$ is $n$, then we say that $\omega$ is a \textit{primitive $n$-th root of unity}. In particular, if $\omega \in K$ is a primitive $n$-th root of unity, then $\op{char}(K)$ does not divide $n$ (see e.g. \cite[\S 7]{Mor96}). An extension $L/K$ is called an \textit{$n$-Kummer extension} if $L/K$ is Galois and $\Gal(L/K)$ is abelian with exponent dividing $n$. We use the following theorem characterizing $n$-Kummer extensions with roots of unity. It is adapted from \cite[Theorem 11.4]{Mor96} and its proof.

\begin{theorem}[Characterization of Kummer extensions]\label{Kummer-extensions}
    Let $K$ be a field containing a primitive $n$-th root of unity and let $L$ be a finite extension of $K$. Then $L/K$ is an $n$-Kummer extension if and only if $L=K(\sqrt[\leftroot{-3}\uproot{3} n_1]{\alpha_1},\dots,\sqrt[\leftroot{-3}\uproot{3} n_r]{\alpha_r})$ for some $\alpha_1,\dots,\alpha_r \in K$ with $n_1,\dots, n_r$ positive integers dividing $n$. Moreover, if $L/K$ is an $n$-Kummer extension, then by the Fundamental Theorem of Finite Abelian Groups, the Galois group $\Gal(L/K) = \langle \eta_1 \rangle \cdots \langle \eta_r \rangle$ is a product of cyclic groups and if the order of $\eta_i$ is $n_i$, for each $i = 1, \dots, r$, then the fixed field $L_i$ of the subgroup $H_i = \langle \eta_1 \rangle \cdots \langle \eta_{i-1} \rangle \langle \eta_{i+1} \rangle \cdots \langle \eta_r \rangle$ is a cyclic extension of $K$ with $\Gal(L_i/K) = \langle \eta_i \rangle$ and there exists an element $\alpha_i \in K$ such that $L_i = F(\sqrt[\leftroot{-3}\uproot{3} n_i]{\alpha_i})$. Then $L=K(\sqrt[\leftroot{-3}\uproot{3} n_1]{\alpha_1},\dots,\sqrt[\leftroot{-3}\uproot{3} n_r]{\alpha_r})$.
\end{theorem}

Then, we have the following embedding for towers $F \subseteq K \subseteq L$ where $L/K$ is an $n$-Kummer extension.

\begin{theorem}\label{Kummer-embedding}
Let $F \subseteq K \subseteq L$ be a tower of fields with $L/F$ finite and separable. Moreover, suppose that $K$ contains a primitive $n$-th root of unity $\omega$ and that $L/K$ is an $n$-Kummer extension. As in Theorem \ref{Kummer-extensions}, write $L = K(\sqrt[\leftroot{-3}\uproot{3} n_1]{\alpha_1}, \dots, \sqrt[\leftroot{-3}\uproot{3} n_r]{\alpha_r})$ for some $\alpha_i \in K$ and some positive integers $n_i$ dividing $n$, and put $\Gal(L/K) = \langle \eta_1 \rangle \cdots \langle \eta_r \rangle$, where the order of $\eta_i$ equals $n_i$ for each $i\in \{1,2,\dots, r\}$. Denote the Galois closures of $K/F$ and $L/F$ by $K^c$ and $L^c$, respectively and let $\Omega = \op{Hom}_{F}(K, K^c)$ be the set of $F$-homomorphisms $K \hookrightarrow K^c$. The group $\Gal(K^c/F)$ acts on $\Omega$ on the left by composition of maps. Then the Galois group $\Gal(L^c/F)$ embeds into the wreath product $\Gal(L/K) \wr_{\Omega} \Gal(K^c/F)$. More precisely, consider the map 
	\begin{align}\label{Galois-embedding-kummer}
		\varphi \colon &\Gal(L^c/F) \longrightarrow \Gal(L/K) \wr_{\Omega} \Gal(K^c/F) \\
		&\quad \quad \qquad  \rho \longmapsto (\sigma_{\rho}, \varepsilon(\rho)) \notag
	\end{align}
	where $\sigma_{\rho} \colon \op{Hom}_F(K, K^c) \longrightarrow \Gal(L/K)$ is defined for every embedding $\tau \in \op{Hom}_F(K, K^c)$ by
	$$
	   \sigma_{\rho}(\tau) = \eta_1^{\chi(\rho, \tau,\alpha_1)}\circ \cdots \circ \eta_r^{\chi(\rho, \tau,\alpha_r)},
	$$
with $\chi(\rho, \tau, \alpha_j) \in \Z/n_j\Z$ chosen to satisfy 
    $$
      \tau(\omega_j)^{\chi(\rho, \tau, \alpha_j)} = \frac{\rho\left(\sqrt[\leftroot{-3}\uproot{3} n_j]{\rho^{-1}(\tau(\alpha_j))}\right)}{\sqrt[\leftroot{-3}\uproot{3} n_j]{\tau(\alpha_j)}},
    $$
where $\omega_j = \omega^{n/n_j}$ for $j \in \{1,2,\cdots,r\}$. Then the map $\varphi$ is an injective group homomorphism.
\end{theorem}

\begin{remark}
    Theorem \ref{Kummer-embedding} indeed gives a sharper embedding than the one from Theorem \ref{Main-Theorem-2}. In fact, as we showed in Remark \ref{Size-Remark}, the size of the wreath product in Theorem \ref{Main-Theorem-2} is equal to $|\Gal(L^c/K)|^{[K:F]} |\Gal(K^c/F)|$, while the size of the wreath product in Theorem \ref{Kummer-embedding} is  given by $|\Gal(L/K)|^{[K:F]} |\Gal(K^c/F)|$, which will be much smaller whenever $L/F$ is a non-Galois extension. In particular, in Example \ref{Kummer-example} we will give an instance of a 6-Kummer extension $L/K$ of degree 72 for which the latter wreath product is about 40 million times smaller than the first one!
\end{remark}

Now, recall that a field extension $L/K$ is said to be \textit{cyclic} if it is Galois and $\Gal(L/K)$ is cyclic. Then, since a cyclic extension of degree $n$ is an $n$-Kummer extension, we immediately obtain the following corollary.

\begin{corollary}\label{Cyclic-embedding}
Let $F \subseteq K \subseteq L$ be a tower of fields with $L/F$ finite and separable. Moreover, suppose that $K$ contains a primitive $n$-th root of unity $\omega$, that $L/K$ is a cyclic extension of degree $n$ and write $L = K(\sqrt[n]{\alpha})$ for some $\alpha \in K$. Denote the Galois closures of $K/F$ and $L/F$ by $K^c$ and $L^c$, respectively, let $\Gal(L/K) = \langle \eta \rangle = \{ \op{id}_L, \eta, \dots, \eta^{n-1} \} \simeq \Z/n\Z$, and let $\Omega = \op{Hom}_{F}(K, K^c)$ be the set of $F$-homomorphisms $K \hookrightarrow K^c$. The group $\Gal(K^c/F)$ acts on $\Omega$ on the left by composition of maps. Then the Galois group $\Gal(L^c/F)$ embeds into the wreath product $\Gal(L/K) \wr_{\Omega} \Gal(K^c/F)$. More precisely, consider the map 
	\begin{align}\label{Galois-embedding-cyclic}
		\varphi \colon &\Gal(L^c/F) \longrightarrow \Gal(L/K) \wr_{\Omega} \Gal(K^c/F) \\
		&\quad \quad \qquad  \rho \longmapsto (\sigma_{\rho}, \varepsilon(\rho)) \notag
	\end{align}
	where $\sigma_{\rho} \colon \op{Hom}_F(K, K^c) \longrightarrow \Gal(L/K)$ is defined for every embedding $\tau \in \op{Hom}_F(K, K^c)$ by
	$$
	   \sigma_{\rho}(\tau) = \eta^{\chi(\rho, \tau)},
	$$
with $\chi(\rho, \tau) \in \Z/n\Z$ chosen to satisfy 
    $$
      \tau(\omega)^{\chi(\rho, \tau)} = \frac{\rho\left(\sqrt[n]{\rho^{-1}(\tau(\alpha))}\right)}{\sqrt[n]{\tau(\alpha)}}.
    $$
	Then the map $\varphi$ is an injective group homomorphism.

\end{corollary}

In particular, an important special case of the previous corollary is the following.

\begin{corollary}\label{Quadratic-embedding}
Let $F \subseteq K \subseteq L$ be a tower of fields with $L/F$ finite and separable and $\op{char}(F) \neq 2$. Moreover, suppose that $L/K$ is a quadratic extension and write $L = K(\sqrt{\alpha})$ for some $\alpha \in K \smallsetminus K^2$. Denote the Galois closures of $K/F$ and $L/F$ by $K^c$ and $L^c$, respectively, let $\Gal(L/K) = \{ \op{id}_L, \eta \}$, and let $\Omega = \op{Hom}_{F}(K, K^c)$ be the set of $F$-homomorphisms $K \hookrightarrow K^c$. The group $\Gal(K^c/F)$ acts on $\Omega$ on the left by composition of maps. Then the Galois group $\Gal(L^c/F)$ embeds into the wreath product $\Gal(L/K) \wr_{\Omega} \Gal(K^c/F)$. More precisely, consider the map 
	\begin{align}\label{Galois-embedding-2}
		\varphi \colon &\Gal(L^c/F) \longrightarrow \Gal(L/K) \wr_{\Omega} \Gal(K^c/F) \\
		&\quad \quad \qquad  \rho \longmapsto (\sigma_{\rho}, \varepsilon(\rho)) \notag
	\end{align}
	where $\sigma_{\rho} \colon \op{Hom}_F(K, K^c) \longrightarrow \Gal(L/K)$ is defined for every embedding $\tau \in \op{Hom}_F(K, K^c)$ by
	$$
	   \sigma_{\rho}(\tau) = \eta^{\chi(\rho, \tau)},
	$$
with $\chi(\rho, \tau) \in \Z/2\Z$ chosen to satisfy 
    $$
      (-1)^{\chi(\rho, \tau)} = \frac{\rho\left(\sqrt{\rho^{-1}(\tau(\alpha))}\right)}{\sqrt{\tau(\alpha)}} \in \{\pm 1 \}.
    $$
	Then the map $\varphi$ is an injective group homomorphism.

\end{corollary}

\begin{remark}
We note that Corollary \ref{Quadratic-embedding} gives a sharper version of \cite[Lemma 2.8]{CO12} (see Theorem \ref{Chai-Oort-Theorem}) and a more general version of \cite[Proposition 3.1]{BSMT17} (see Theorem \ref{Barquero-Sanchez-Masri-Thorne-Theorem}). In particular, as we will show in Section \ref{Technical-section}, we can use Corollary \ref{Quadratic-embedding} in combination with the technical results for identifying wreath products with isomorphic components from that section, to deduce these two theorems.
\end{remark}




\subsection{Organization of the paper}

The paper is organized as follows. In Section \ref{Applications} we briefly survey some of the recent applications of embedding theorems for Galois groups into wreath products. In Section \ref{Wreath-section} we give the basic definitions and notation for wreath products and in Section \ref{Wreath-properties-section} we prove some results about isomorphisms and subgroups of wreath products which are useful in identifying wreath products of Galois groups when used in combination with our embedding results. Then, in Section \ref{Kaloujnine-Krasner-section} we explain the extension problem for groups and give a detailed proof of the Kaloujnine-Krasner Universal Embedding Theorem, from which we deduce the explicit embedding of Theorem \ref{Main-Theorem}. In Section \ref{Main-Theorem-2-section} we prove Theorem \ref{Main-Theorem-2} and in Section \ref{Kummer-section} we prove Theorem \ref{Kummer-embedding}. Finally, in Section \ref{Examples-section} we give some detailed examples of our main theorems.

\subsection{Experimental data on the sizes of wreath products}\label{tables-and-graphs-intro}

In this subsection we give some numerical computations of the sizes of the wreath products appearing in our main theorems and we compare them in order to highlight the relative differences in their sizes. The data contained in Table \ref{Wreath-size-table} and in Figures \ref{plot_K_F_3}, \ref{plot_K_F_4} and \ref{plot_K_F_5} is explained in Remark \ref{Size-Remark}.

Moreover, in Figures \ref{plot_K_F_3}, \ref{plot_K_F_4} and \ref{plot_K_F_5}, we plot the logarithms of the sizes of the wreath products appearing in Theorems \ref{Main-Theorem} and \ref{Main-Theorem-2}, namely $\log(|\Gal(L^c/K^c) \wr_r \Gal(K^c/F)|)$ (in red), and $\log(|\Gal(L^c/K) \wr_{\Omega} \Gal(K^c/F)|)$ (in blue), respectively, as functions of $m=[L^c:F]$, in the cases in which $K/F$ is non-Galois and $2 \leq [K:F] \leq 5$. Additionally, as was observed in Remark \ref{Size-Remark}, the regular wreath product from Theorem \ref{Main-Theorem} is strictly bigger than the wreath product from Theorem \ref{Main-Theorem-2} when $m \geq 2[K^c : F]$, except in the cases in which $\Gal(K^c/F)$ is dihedral. Thus, we have highlited the vertical line $m = 2[K^c:F]$ in the plots.

\begin{table}[H]
{\renewcommand{\arraystretch}{2.2}
\begin{tabular}{@{}ccccc@{}}
\toprule
$[K:F]$ & $[K^c:F]$ & $\Gal(K^c/F)$                                   & $|\Gal(L^c/K^c) \wr_r \Gal(K^c/F)|$                    & $|\Gal(L^c/K) \wr_{\Omega} \Gal(K^c/F)|$ \\ \midrule
2                  & 2  & $C_2$            & $\dfrac{m^2}{2}$                    & $\dfrac{m^2}{2}$    \\ \midrule
\multirow{2}{*}{3} & 3  & $C_3$            & $\dfrac{m^3}{9}$                    & $\dfrac{m^3}{9}$    \\
                   & 6  & $S_3$            & $\dfrac{m^6}{6^5}$                  & $\dfrac{2m^3}{9}$   \\ \midrule
\multirow{5}{*}{4} & 4  & $C_4$            & $\dfrac{m^4}{64}$                   & $\dfrac{m^4}{64}$   \\
                   & 4  & $C_2 \times C_2$ & $\dfrac{m^4}{64}$                   & $\dfrac{m^4}{64}$   \\
                   & 8  & $D_4$            & $\dfrac{m^8}{2^{21}}$               & $\dfrac{m^4}{2^5}$  \\
                   & 12 & $A_4$            & $\dfrac{m^{12}}{2^{22} \cdot 3^{11}}$ & $\dfrac{3m^4}{64}$  \\
                   & 24 & $S_4$            & $\dfrac{m^{24}}{2^{69} \cdot 3^{23}}$ & $\dfrac{3m^4}{32}$  \\ \midrule
\multirow{5}{*}{5} & 5  & $C_5$            & $\dfrac{m^5}{625}$                  & $\dfrac{m^5}{625}$  \\
                   & 10 & $D_5$            & $\dfrac{m^{10}}{2^{9} \cdot 5^{9}}$ & $\dfrac{2m^5}{625}$ \\
        & 20      & $F_5 \simeq \operatorname{AGL}(1, \mathbb{F}_5)$ & $\dfrac{m^{20}}{2^{38} \cdot 5^{19}}$                  & $\dfrac{4m^5}{625}$                      \\
        & 60      & $A_5$                                           & $\dfrac{m^{60}}{2^{118} \cdot 3^{59} \cdot 5^{59}}$    & $\dfrac{12m^5}{625}$                      \\
        & 120     & $S_5$                                           & $\dfrac{m^{120}}{2^{357} \cdot 3^{119} \cdot 5^{119}}$ & $\dfrac{24m^5}{625}$                     \\ \bottomrule
\end{tabular}
}
\caption{Comparison of the sizes of the wreath products appearing in Theorems \ref{Main-Theorem} and \ref{Main-Theorem-2}. Here we have given the formulas in terms of $m := |\Gal(L^c/F)| = [L^c:F]$. In particular, we always have $[K^c:F]$ divides $m$, so the fractions appearing in the table are integers.}
\label{Wreath-size-table}
\end{table}

\begin{remark}
Note than in each row of Table \ref{Wreath-size-table} we have $m \geq 2[K^c:F]$ for a nontrivial tower of fields $F \subsetneq K \subsetneq L$. More details about the information of this table are given in Remarks \ref{Same-group-remark} and \ref{Size-Remark}. Moreover, for the definition of the one dimensional affine linear group $\operatorname{AGL}(1, \mathbb{F}_5)$ see e.g. Section \ref{Solvability-application}.
\end{remark}

\begin{figure}[H]
    \begin{subfigure}[b]{0.45\textwidth}
        \begin{tikzpicture}[scale=0.8]
            \begin{axis}[
                legend pos=south east,
                xlabel={$m$},
                title={$\Gal(K^c/F) \simeq S_3 \quad [K^c:F]=6$},
                xtick = {6,12,24,36,48,60},]
                \addplot[
                    scatter,only marks,scatter src=explicit symbolic,
                    scatter/classes={
                    omega={GaloisBlue},
                    regular={GaloisRed}
                    },
                ]
                table[x=x,y=y,meta=label]{
                    x    y    label
                    6	1.791759469228055	 regular
                    12	5.950642552587727	 regular
                    18	8.383433201236713	 regular
                    24	10.1095256359474	 regular
                    30	11.448386943832658	 regular
                    36	12.542316284596385	 regular
                    42	13.467220363559935	 regular
                    48	14.26840871930707	 regular
                    54	14.975106933245371	 regular
                    60	15.60727002719233	 regular
                    6	3.871201010907891	 omega
                    12	5.950642552587727	 omega
                    18	7.16703787691222	 omega
                    24	8.030084094267563	 omega
                    30	8.699514748210191	 omega
                    36	9.246479418592056	 omega
                    42	9.708931458073831	 omega
                    48	10.1095256359474	 omega
                    54	10.462874742916549	 omega
                    60	10.778956289890028	 omega
                };
                \draw [dashed, very thick, color=gray] (12,-1) -- (12,20);
            \end{axis}
        \end{tikzpicture}
     \end{subfigure}
     \caption{The sizes of the wreath products when $[K:F]=3$.}
     \label{plot_K_F_3}
\end{figure}
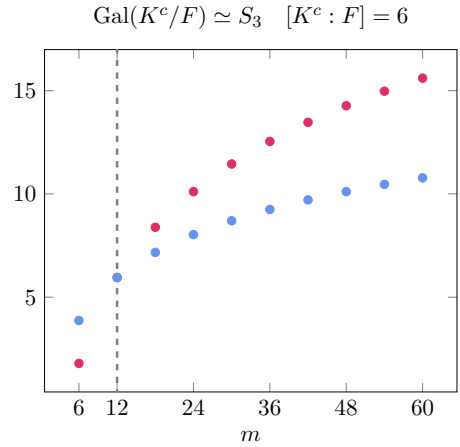

\begin{figure}[H]
    \begin{subfigure}[b]{0.45\textwidth}
        \begin{tikzpicture}[scale=0.8]
            \begin{axis}[
                legend pos=south east,
                xlabel={$m$},
                title={$\Gal(K^c/F) \simeq D_4 \quad [K^c:F]=8$},
                xtick = {8,16,32,48,64,80}]
                \addplot[
                    scatter,only marks,scatter src=explicit symbolic,
                    scatter/classes={
                        omega={GaloisBlue},
                        regular={GaloisRed}
                    },
                ]
                table[x=x,y=y,meta=label]{
                    x    y    label
                    8	2.0794415416798357	 regular
                    16	7.6246189861593985	 regular
                    24	10.868339851024713	 regular
                    32	13.16979643063896	 regular
                    40	14.954944841152638	 regular
                    48	16.413517295504278	 regular
                    56	17.646722734122342	 regular
                    64	18.714973875118524	 regular
                    72	19.65723816036959	 regular
                    80	20.5001222856322	 regular
                    8	4.852030263919617	 omega
                    16	7.6246189861593985	 omega
                    24	9.246479418592056	 omega
                    32	10.39720770839918	 omega
                    40	11.289781913656018	 omega
                    48	12.019068140831838	 omega
                    56	12.63567086014087	 omega
                    64	13.16979643063896	 omega
                    72	13.640928573264494	 omega
                    80	14.0623706358958	 omega
                };
                 \draw [dashed, very thick, color=gray] (16,-1) -- (16,26);
            \end{axis}
        \end{tikzpicture}
     \end{subfigure}
    \begin{subfigure}[b]{0.45\textwidth}
        \begin{tikzpicture}[scale=0.8]
            \begin{axis}[
                legend pos=south east,
                xlabel={$m$},
                title={$\Gal(K^c/F) \simeq A_4 \quad [K^c:F]=12$},
                xtick = {12,24,48,72,96,120}]
                \addplot[
                    scatter,only marks,scatter src=explicit symbolic,
                    scatter/classes={
                        omega={GaloisBlue},
                        regular={GaloisRed}
                    },
                ]
                table[x=x,y=y,meta=label]{
                    x    y    label
                    12	2.4849066497880004	 regular
                    24	10.802672816507345	 regular
                    36	15.668254113805316	 regular
                    48	19.120438983226688	 regular
                    60	21.798161598997204	 regular
                    72	23.98602028052466	 regular
                    84	25.83582843845176	 regular
                    96	27.438205149946032	 regular
                    108	28.85160157782263	 regular
                    120	30.11592776571655	 regular
                    12	6.879355804460439	 omega
                    24	9.65194452670022	 omega
                    36	11.273804959132878	 omega
                    48	12.424533248940001	 omega
                    60	13.31710745419684	 omega
                    72	14.04639368137266	 omega
                    84	14.662996400681692	 omega
                    96	15.197121971179783	 omega
                    108	15.668254113805316	 omega
                    120	16.08969617643662	 omega
                };
                 \draw [dashed, very thick, color=gray] (24,-1) -- (24,38);
            \end{axis}
        \end{tikzpicture}
     \end{subfigure}
    \begin{subfigure}[b]{0.45\textwidth}
        \begin{tikzpicture}[scale=0.8]
            \begin{axis}[
                legend pos=south east,
                xlabel={$m$},
                title={$\Gal(K^c/F) \simeq S_4 \quad [K^c:F]=24$},
                xtick = {24,48,96,144,192,240}]
                \addplot[
                    scatter,only marks,scatter src=explicit symbolic,
                    scatter/classes={
                        omega={GaloisBlue},
                        regular={GaloisRed}
                    },
                ]
                table[x=x,y=y,meta=label]{
                    x    y    label
                    24	3.1780538303479458	 regular
                    48	19.81358616378663	 regular
                    72	29.54474875838258	 regular
                    96	36.44911849722532	 regular
                    120	41.804563728766354	 regular
                    144	46.180281091821264	 regular
                    168	49.879897407675465	 regular
                    192	53.08465083066401	 regular
                    216	55.91144368641721	 regular
                    240	58.44009606220504	 regular
                    24	10.345091707260165	 omega
                    48	13.117680429499947	 omega
                    72	14.739540861932605	 omega
                    96	15.890269151739728	 omega
                    120	16.782843356996565	 omega
                    144	17.512129584172385	 omega
                    168	18.128732303481417	 omega
                    192	18.662857873979508	 omega
                    216	19.133990016605043	 omega
                    240	19.555432079236347	 omega
                };
                 \draw [dashed, very thick, color=gray] (48,-1) -- (48,64);
            \end{axis}
        \end{tikzpicture}
     \end{subfigure}
    \caption{The sizes of the wreath products when $[K:F]=4$.}
    \label{plot_K_F_4}
\end{figure}
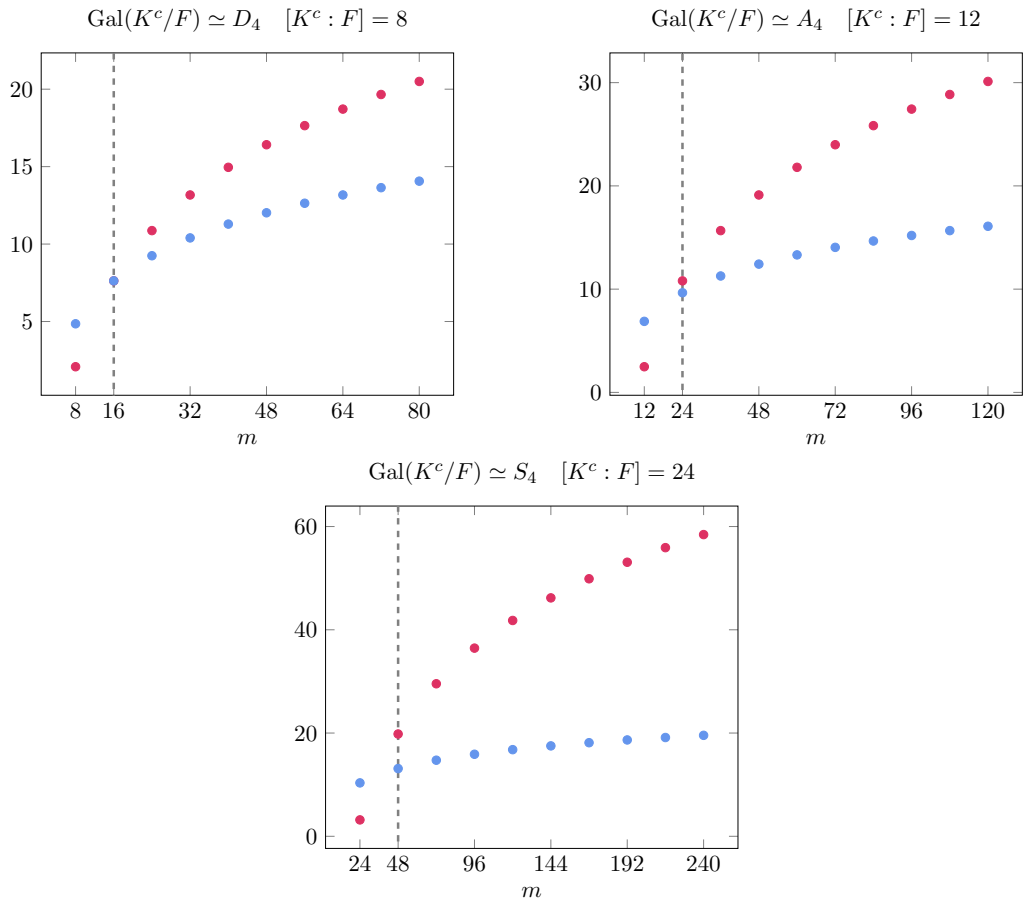

\begin{figure}[H]
\hspace{2mm}
    \begin{subfigure}[b]{0.45\textwidth}
        \begin{tikzpicture}[scale=0.8]
            \begin{axis}[
                legend pos=south east,
                xlabel={$m$},
                title={$\Gal(K^c/F) \simeq D_5 \quad [K^c:F]=10$},
                xtick = {10,20,60,80,100}]
                \addplot[
                    scatter,only marks,scatter src=explicit symbolic,
                    scatter/classes={
                        omega={GaloisBlue},
                        regular={GaloisRed}
                    },
                ]
                table[x=x,y=y,meta=label]{
                    x    y    label
                    10	2.302585092994046	 regular
                    20	9.234056898593499	 regular
                    30	13.288707979675143	 regular
                    40	16.16552870419295	 regular
                    50	18.39696421733505	 regular
                    60	20.220179785274595	 regular
                    70	21.761686583547178	 regular
                    80	23.097000509792405	 regular
                    90	24.27483086635624	 regular
                    100	25.328436022934504	 regular
                    10	5.768320995793772	 omega
                    20	9.234056898593499	 omega
                    30	11.261382439134321	 omega
                    40	12.699792801393226	 omega
                    50	13.815510557964275	 omega
                    60	14.727118341934048	 omega
                    70	15.49787174107034	 omega
                    80	16.16552870419295	 omega
                    90	16.75444388247487	 omega
                    100	17.281246460764002	 omega
                };
                 \draw [dashed, very thick, color=gray] (20,-1) -- (20,36);
            \end{axis}
        \end{tikzpicture}
     \end{subfigure}
    \begin{subfigure}[b]{0.45\textwidth}
        \begin{tikzpicture}[scale=0.8]
            \begin{axis}[
                legend pos=south east,
                xlabel={$m$},
                title={$\Gal(K^c/F) \simeq F_5 \simeq \operatorname{AGL}(1, \mathbb{F}_5) \quad [K^c:F]=20$},
                xtick = {20,40,80,120,160,200}]
                \addplot[
                    scatter,only marks,scatter src=explicit symbolic,
                    scatter/classes={
                        omega={GaloisBlue},
                        regular={GaloisRed}
                    },
                ]
                table[x=x,y=y,meta=label]{
                    x    y    label
                    20	2.995732273553991	 regular
                    40	16.8586758847529	 regular
                    60	24.967978046916183	 regular
                    80	30.721619495951803	 regular
                    100	35.184490522236	 regular
                    120	38.83092165811509	 regular
                    140	41.91393525466026	 regular
                    160	44.58456310715071	 regular
                    180	46.940223820278376	 regular
                    200	49.047434133434905	 regular
                    20	9.927204079153444	 omega
                    40	13.392939981953171	 omega
                    60	15.420265522493992	 omega
                    80	16.8586758847529	 omega
                    100	17.974393641323946	 omega
                    120	18.88600142529372	 omega
                    140	19.65675482443001	 omega
                    160	20.324411787552624	 omega
                    180	20.913326965834543	 omega
                    200	21.44012954412367	 omega
                };
                 \draw [dashed, very thick, color=gray] (40,-1) -- (40,68);
            \end{axis}
        \end{tikzpicture}
     \end{subfigure}
    \begin{subfigure}[b]{0.45\textwidth}
        \begin{tikzpicture}[scale=0.8]
            \begin{axis}[
                legend pos=south east,
                xlabel={$m$},
                title={$\Gal(K^c/F) \simeq A_5 \quad [K^c:F]=60$},
                xtick = {60,120,240,360,480,600}]
                \addplot[
                    scatter,only marks,scatter src=explicit symbolic,
                    scatter/classes={
                        omega={GaloisBlue},
                        regular={GaloisRed}
                    },
                ]
                table[x=x,y=y,meta=label]{
                    x    y    label
                    60	4.0943445622221	 regular
                    120	45.68317539581882	 regular
                    180	70.01108188230869	 regular
                    240	87.27200622941554	 regular
                    300	100.66061930826812	 regular
                    360	111.5999127159054	 regular
                    420	120.8489535055409	 regular
                    480	128.86083706301227	 regular
                    540	135.92781920239526	 regular
                    600	142.24945014186486	 regular
                    60	16.518877811162103	 omega
                    120	19.984613713961828	 omega
                    180	22.01193925450265	 omega
                    240	23.450349616761557	 omega
                    300	24.566067373332604	 omega
                    360	25.47767515730238	 omega
                    420	26.248428556438668	 omega
                    480	26.916085519561282	 omega
                    540	27.5050006978432	 omega
                    600	28.03180327613233	 omega
                };
                 \draw [dashed, very thick, color=gray] (120,-1) -- (120,160);
            \end{axis}
        \end{tikzpicture}
     \end{subfigure}
    \begin{subfigure}[b]{0.45\textwidth}
        \begin{tikzpicture}[scale=0.8]
            \begin{axis}[
                legend pos=south east,
                xlabel={$m$},
                title={$\Gal(K^c/F) \simeq S_5 \quad [K^c:F]=120$},
                xtick = {120,240,480,720,960,1200}]
                \addplot[
                    scatter,only marks,scatter src=explicit symbolic,
                    scatter/classes={
                        omega={GaloisBlue},
                        regular={GaloisRed}
                    },
                ]
                table[x=x,y=y,meta=label]{
                    x    y    label
                    120	4.787491742782046	 regular
                    240	87.96515340997549	 regular
                    360	136.6209663829552	 regular
                    480	171.14281507716893	 regular
                    600	197.92004123487408	 regular
                    720	219.79862805014864	 regular
                    840	238.29670962941964	 regular
                    960	254.32047674436237	 regular
                    1080	268.45444102312837	 regular
                    1200	281.0977029020675	 regular
                    120	20.677760894521775	 omega
                    240	24.1434967973215	 omega
                    360	26.170822337862322	 omega
                    480	27.609232700121225	 omega
                    600	28.724950456692277	 omega
                    720	29.636558240662048	 omega
                    840	30.40731163979834	 omega
                    960	31.074968602920954	 omega
                    1080	31.66388378120287	 omega
                    1200	32.190686359492005	 omega
                };
                 \draw [dashed, very thick, color=gray] (240,-1) -- (240,320);
            \end{axis}
        \end{tikzpicture}
     \end{subfigure}
    \caption{The sizes of the wreath products when $[K:F]=5$.}
    \label{plot_K_F_5}
\end{figure}
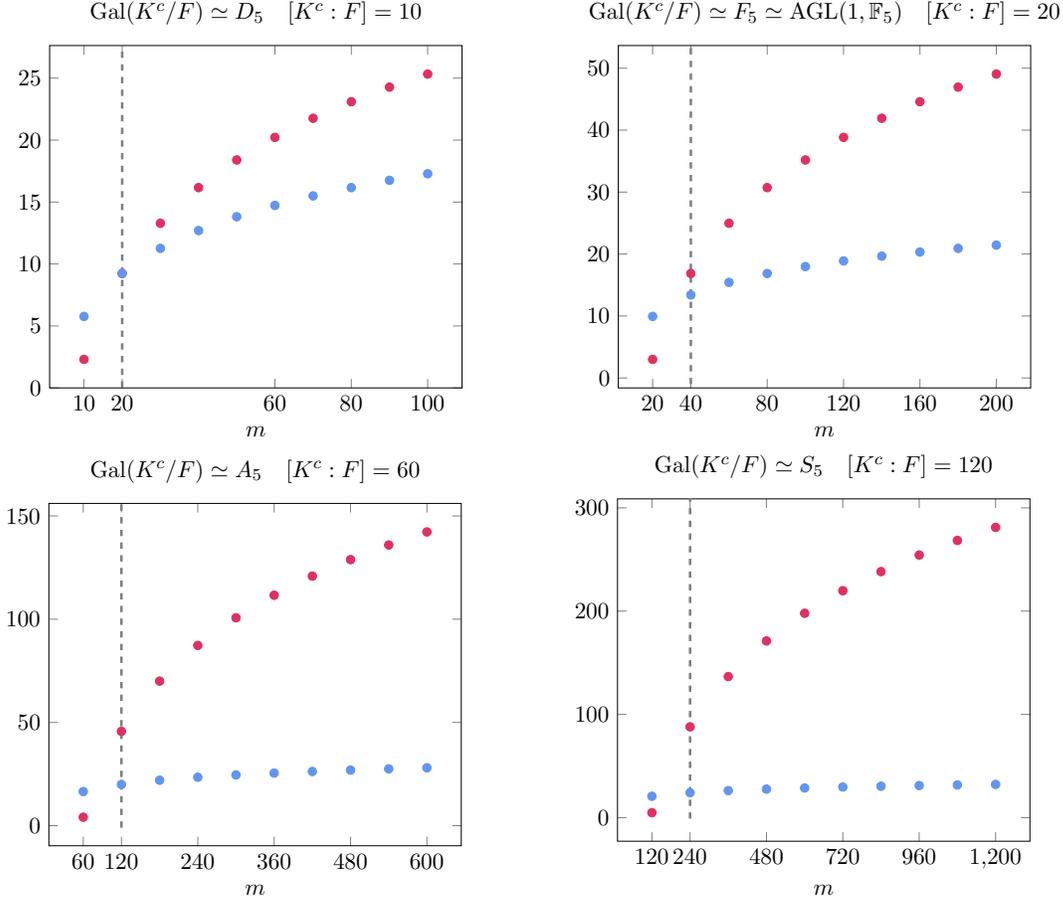

\section{Some applications of the embedding of Galois groups into wreath products}\label{Applications}

In this section we illustrate with some explicit examples, instances where embedding theorems for Galois groups into wreath products like the ones we prove in the paper, and also some instances for Galois groups that are equal to certain wreath products, have been used in recent applications in field theory, arithmetic statistics, number theory and arithmetic geometry.

The background required for each application increases as we move from one to the next. We have made an effort to give enough details for the reader to get a reasonable, if superficial, understanding of each application, and we have given several references so that the interested reader can read further on each one. Nevertheless, naturally we won't be able to define every term, so the reader who doesn't already know some of the concepts that are left undefined in the applications, will have to take them as black boxes and consult the references for the definitions and details.


\subsection{Application to the solvability by radicals of imprimitive polynomials}\label{Solvability-application}

For the details of the material described in this subsection the reader is referred to \cite[\S 6.4 and \S 14.2]{Cox12}.

Let $p$ be a prime and consider the finite field $\mathbb{F}_p$. For $a, b \in \mathbb{F}_p$, let $\gamma_{a, b} \colon \mathbb{F}_p \to \mathbb{F}_p$ be the affine linear transformation defined by $\gamma_{a, b}(t) = at + b$ for every $t \in \mathbb{F}_p$. Such an affine linear transformation is bijective if and only if $a \neq 0$. Then, the \textit{one dimensional affine linear group over $\mathbb{F}_p$} is the group defined by

$$
\operatorname{AGL}(1, \mathbb{F}_p) := \{ \gamma_{a, b} \suchthat (a, b) \in \mathbb{F}_p^{\times} \times \mathbb{F}_p \},
$$
with group operation given by composition of functions. This group has order $p(p-1)$ and moreover $\operatorname{AGL}(1, \mathbb{F}_p) \simeq \mathbb{F}_p \rtimes_{\theta} \mathbb{F}_p^{\times}$, where  $\theta \colon \mathbb{F}_{p}^{\times} \to \operatorname{Aut}(\mathbb{F}_p)$ is given, for every $c \in \mathbb{F}_p^{\times}$, by $\theta_c(t) = ct$ for every $t \in \mathbb{F}_p$. An explicit isomorphism is given by $\psi \colon \operatorname{AGL}(1, \mathbb{F}_p) \longrightarrow \mathbb{F}_p \rtimes_{\theta} \mathbb{F}_p^{\times}$, where $\psi(\gamma_{a, b}) = (b, a)$ for every $\gamma_{a, b} \in \operatorname{AGL}(1, \mathbb{F}_p)$.

Now, let $G$ be a subgroup of the symmetric group $S_n$. Then $G$ is said to be \textit{imprimitive} if there exists a partition 
$$
\{ 1, \dots, n \} = R_1 \cup \cdots \cup R_k
$$
with $k > 1$ and $|R_i| > 1$ for at least one $i \in \{ 1, \dots, k \}$, such that for every $\tau \in G$ and every $i \in \{ 1, \dots, k \}$ we have $\tau(R_i) = R_j$ for some $j \in \{ 1, \dots, k\}$. Moreover, if no such partition exists, then $G$ is said to be \textit{primitive}.

In particular, this terminology can also be used for polynomials. More precisely, if $f \in F[x]$ is a separable polynomial over a field $F$, then $f$ is said to be \textit{imprimitive} if its Galois group $\Gal(f)$ over $F$, thought as a permutation group on the roots of $f$, is imprimitive. Similarly, if its Galois group is not imprimitive, the polynomial $f$ is said to be \textit{primitive}.

\begin{remark}\label{imprimitive-wreath}
Let $n \geq 4$ be composite and let $n = \ell k$ be a nontrivial factorization. Then, the wreath product $S_{\ell} \wr_{\Omega} S_{k}$, where $\Omega:= \{ 1, \dots, k \}$ and $S_k$ acts on $\Omega$ as in Remark \ref{tuple-remark}, is an example of an imprimitive subgroup of $S_n$. For this, see e.g. \cite[\S 14.2 B]{Cox12}.
\end{remark}

Then, the solvability by radicals of an imprimitive polynomial of degree $p^2$ is characterized as follows (see e.g. \cite[Corollary 14.2.16]{Cox12}).

\begin{theorem}\label{Imprimitive-Theorem}
    Let $F$ be a field of characteristic 0, let $f(x) \in F[x]$ be an irreducible imprimitive polynomial of degree $p^2$, and let $L$ be a splitting field for $f$ over F. Then $f$ is solvable by radicals over $F$ if and only if the Galois group $\Gal(L/F)$ of $f$ over $F$ embeds into the wreath product $\operatorname{AGL}(1, \mathbb{F}_p) \wr_{\Omega} \operatorname{AGL}(1, \mathbb{F}_p)$, where $\Omega := \mathbb{F}_p$ and $\operatorname{AGL}(1, \mathbb{F}_p)$ acts on $\Omega$ by evaluation, i.e., for every $t \in \Omega$ and every $\gamma_{a, b} \in \operatorname{AGL}(1, \mathbb{F}_p)$ we have $\gamma_{a, b} \cdot t := \gamma_{a, b}(t) = at+b$.
\end{theorem}

\begin{example}
Consider the polynomial $f(x) = x^9 - 3x^8 + 2x^7 + 2x^4 - 2x^2 - 1 \in \Q[x]$. This corresponds to the entry \cite[\href{https://www.lmfdb.org/NumberField/9.1.74839217.1}{Number field 9.1.74839217.1}]{LMFDB} in the \texttt{LMFDB} database. Then $f$ is irreducible over $\Q$ and its Galois group over $\Q$ is $\Gal(f) \simeq S_3 \wr_{\widehat{\Omega}} S_3$, where $\widehat{\Omega} := \{ 1, 2, 3 \}$ is an $S_3$-set with the natural action by permutations. Hence, by Remark \ref{imprimitive-wreath}, the polynomial $f$ is imprimitive. Note that $|\Gal(f)| = |S_3 \wr_{\widehat{\Omega}} S_3| = |S_3|^{\widehat{\Omega}} \cdot |S_3| = 6^3 \cdot 6 = 1296$. Moreover, we observe that $\operatorname{AGL}(1, \mathbb{F}_3)$ is a non-abelian group of order $6$, and hence $\operatorname{AGL}(1, \mathbb{F}_3) \simeq S_3$. Therefore, by Proposition \ref{Omega-wreath-isomorphism}, we have $\Gal(f) \simeq \operatorname{AGL}(1, \mathbb{F}_3) \wr_{\Omega} \operatorname{AGL}(1, \mathbb{F}_3)$ with $\Omega := \mathbb{F}_3$. Then, by Theorem \ref{Imprimitive-Theorem}, we conclude that the polynomial $f$ is solvable by radicals.
\end{example}

\subsection{Application to the distribution of number fields and the Malle Conjectures}

In 2002 and 2004, G. Malle \cite{Mal02, Mal04} made two very precise conjectures regarding the density of number fields with a given Galois group. Specifically, for a number field $K$, a finite transitive permutation group $G \leq S_n$ and a positive real number $x$, he defined the counting function
$$
N(K,G;x) = | \{L/K \suchthat \Gal(L^c/K) \simeq G \text{ and } \left| \sN_{K/\Q}(d(L/K)) \right| \leq x \} |,
$$
which counts the number of field extensions $L$ of $K$ of degree $n$, whose Galois closure $L^c$ has Galois group over $K$ isomorphic to the given group $G$, and such that the absolute value of the norm of their relative discriminant is bounded by $x$. 

For any permutation group $G \leq S_n$, Malle defines a constant $0 \leq a(G) \leq 1$, which is a rational number that depends only on the group $G$ itself and is explicitly defined in terms of the indices of the permutations in $G$ (see e.g. \cite[p. 316]{Mal02}).

Then, Malle's first conjecture describes the growth of the counting function $N(K, G;x)$ by asserting that for all $\epsilon>0$ there exist positive constants $c_1(K,G)$ and $c_2(K,G;\epsilon)$ such that
\begin{align}\label{malle-conjecture-1}
c_1(K,G)x^{a(G)} \leq N(K,G;x) \leq c_2(K,G;\epsilon) x^{a(G) + \epsilon}.
\end{align}

 Furthermore, in \cite[Sections 4--5]{Mal02}, Malle studies the \textit{inductive consistency} of his conjecture \eqref{malle-conjecture-1} with the formation of direct products and wreath products of groups for which conjecture \eqref{malle-conjecture-1} holds. In particular, for wreath products, Malle proves the following result.

\begin{proposition}[\cite{Mal02}, Corollary 5.3]\label{prop-wreath-malle-conjecture-1}
Assume that the upper bound in (\ref{malle-conjecture-1}) holds for a transitive permutation group $H \leq S_{k}$, and that there exists $\delta_0>0$ such that for a fixed $\epsilon>0$ we have $c_2(L,H, \epsilon) \leq c(H,\epsilon)d_L^{\delta_0}$ for all $L/K$ with $\Gal(L^c/K) \simeq G \leq S_{\ell}$ and where $d_L = |\sN_{K/\Q}(d(L/K))|$. Assume moreover that there exists a constant $e_G$ such that $e_G < ka(H) - \delta_0$ and $$N(K,G;x) \leq c(K,G;\epsilon)x^{e_G+\epsilon}.$$ Then, the upper bound in  \eqref{malle-conjecture-1} holds for the wreath product $G\wr_{\Omega}H$, where $\Omega=\{1,\cdots, k\}$ and $H$ acts on $\Omega$ in the natural way.
\end{proposition}

\begin{remark}
Note that this deals with the upper bound in (\ref{malle-conjecture-1}). Obtaining a lower bound for $N(K,G \wr_{\Omega}H;x)$ is, unfortunately, not that clear. Let $L/K$ be an extension such that $\Gal(L^c/K) \simeq H$ for some transitive subgroup $H \leq S_k$ and
let $G \leq S_{\ell}$ be a transitive subgroup.
Malle briefly mentions in the last paragraph of page 322 in \cite{Mal02} that for any extension $M/L$ with $\Gal(M^c/L) \simeq G$, the Galois group $\Gal(M^c/K)$ can be embedded in the wreath product $$\Gal(M^c/L) \wr_{\operatorname{Hom}_{K}(L, L^c)} \Gal(L^c/K) \simeq G \wr_{\Omega} H,$$ 
where $\Omega := \{ 1, \dots, k \}$ and $H$ acts on $\Omega$ in the natural way. This is a consequence, for example, of the result of Theorem \ref{Main-Theorem-2} and Proposition \ref{Omega-wreath-isomorphism} of this paper. Then, Malle argues in \cite[pp. 322--323]{Mal02} that the lower bound in (\ref{malle-conjecture-1}) can be obtained for the counting function $N(K,G \wr_{\Omega}H;x)$ if the density of extensions $M/L$ for which $\Gal(M^c/L)$ is maximal, is positive; that is, the extensions for which the Galois group $\Gal(M^c/K)$ equals the full wreath product $G \wr_{\Omega} H$. This strategy was first carried out by J. Klüners for the wreath product $C_2 \wr_{\Omega} H$, as is explained in the two following propositions.
\end{remark}

Malle refined his first conjecture \eqref{malle-conjecture-1} in \cite{Mal04}, now providing an asymptotic estimation for the counting function. More explicitly, again for a transitive subgroup $G \leq S_n$ and a number field $K$, Malle defined another constant $b(K,G)$ that only depends on $G$ and $K$ and is defined in terms of conjugacy classes of $G$. Then, the stronger version of the Malle conjectures states that there exists a positive constant $c(K,G)$ such that
\begin{align}\label{malle-conjecture-2}
N(K,G;x) \sim c(K,G) x^{a(G)} (\log x)^{b(K,G)-1}.
\end{align}
Malle further supports this stronger form of the conjecture with heuristic arguments and experimental data. 

Moreover, regarding the topic of this paper, it is important to mention that this stronger form of the conjecture was further explored by J. Klüners for wreath products in \cite{Klu12}, particularly for wreath products of the form $C_2 \wr_{\Omega} H$, with $H$ a subgroup of $S_k$ and $\Omega = \{ 1, \dots, k \}$. For example, Klüners shows that Proposition \ref{prop-wreath-malle-conjecture-1} becomes simpler for wreath products like this, as is stated in the following result.

\begin{proposition}[\cite{Klu12}, Proposition 1]\label{prop-wreath-malle-conjecture-2}
Let $H \leq S_{k}$ be a transitive permutation group and suppose that for every $\delta > 0$, there exist positive constants $c(K, H; \delta)$ such that $N(K, H; x) \leq c(K, H; \delta) x^{1+\delta}$. Then for any $\epsilon > 0$ there exists a constant $c(K, C_2 \wr_{\Omega} H; \epsilon)$ such that
$$
N(K, C_2 \wr_{\Omega} H; x) \leq c(K, C_2 \wr_{\Omega} H; \epsilon) x^{a(C_2\wr_{\Omega}H)+\epsilon}.
$$
Here, $\Omega = \{1,2,\cdots, k\}.$
\end{proposition}

Furthermore, Klüners obtains conditions for conjecture (\ref{malle-conjecture-2}) to hold for wreath products. This is stated in the following result.

\begin{proposition}[\cite{Klu12}, Corollary 5 and Theorem 6]
Assume $H \leq S_{k}$ satisfies that
$$
N(K,H;x) = O_{K,H,\epsilon}(X^{1+\epsilon}),
$$
that is, that for all $\epsilon>0$ there exists a constant $c(K,H;\epsilon)$ such that $N(K,H;x) \leq c(K,H;\epsilon) x^{1+\epsilon}$ for large enough $x$. Then conjecture (\ref{malle-conjecture-2}) is true for all groups of the form $G = C_2 \wr_{\Omega} H$, where $\Omega=\{1,2,\cdots,k\}$.
\end{proposition}

The two decades that have passed since the formulation of the Malle conjectures have seen a continued stream of work towards them. In particular, it is important to point out that Klüners \cite{Klu05} showed that the predicted value of the constant $b(K, G)$ in the strong form of the Malle conjectures \eqref{malle-conjecture-2} is not always correct. Later on S. Türkelli \cite{Tur15} proposed a correction for the exponent of the logarithm in \eqref{malle-conjecture-2}. The interested reader can consult for example the introduction of \cite{Alb21} for a short survey of some of the more relevant advances in this area. Moreover, in that paper, Alberts proposes a refinement of the Malle conjectures and compares it to the correction predicted by Türkelli.

\subsection{Application to the distribution of CM fields and the Colmez conjecture}

In this application, we briefly describe a conjecture of P. Colmez that relates a height function defined on abelian varieties to certain sums of logarithmic derivatives of Artin $L$-functions evaluated at 0. In particular, we describe some recent advances on the conjecture that make use of the embedding of Galois groups into wreath products. We will start by describing the conjecture in the following paragraphs.

A \textit{CM field} $E$ of degree $2n$ is a number field such that $E$ is a totally imaginary quadratic extension of a totally real number field $F$ of degree $n$. In particular, this means that there exists a totally negative element $\Delta \ll 0$ in $F$ such that $E = F(\sqrt{\Delta})$ (see e.g. \cite[Proposition 1.4]{Mil20}). Basic examples of CM fields are given by the cyclotomic fields $E_n = \Q(\zeta_n)$, where $\zeta_n = e^{2\pi i / n}$ is a primitive $n$-th root of unity. In this case, the maximal totally real subfield of $E_n$ is $F_n = \Q(\zeta_n + \zeta_n^{-1}) = \Q(\cos{(2\pi /n)})$. In the quadratic case, the CM fields of degree 2 are just the imaginary quadratic fields $\Q(\sqrt{-d})$, where $d > 0$ is a square-free integer, and the maximal totally real subfield is $\Q$. For a $CM$ field $E$ of degree $2n$, the complex embeddings $E \hooklongrightarrow \C$ come in complex conjugate pairs. Thus, the embeddings can be numbered as $\sigma_1, \overline{\sigma_1}, \dots, \sigma_n, \overline{\sigma_n}$. Then, a \textit{CM type} for $E$ is a set of $n$ embeddings containing exactly one embedding from each complex conjugate pair. Since there are two choices from each pair, the set of CM types for $E$, denoted by $\Phi(E)$, has exactly $2^n$ CM types.

These fields appear naturally in the theory of complex multiplication for abelian varieties and this is the reason for the terminology CM. The theory of complex multiplication was described by D. Hilbert as \textit{``not only the most beautiful part of mathematics but also of the whole of science.''} (see e.g. \cite{Taussky43}). The interested reader can consult for example \cite[Chapter 6]{ST15} and \cite{Cox22} for a very readable treatment of the basic theory of complex multiplication for elliptic curves, or \cite[Chapter 2]{Sil94} for a more advanced and complete treatment, and ultimately \cite{Mil20} or \cite{Lan83} for a more general treatment of the theory for abelian varieties.

An abelian variety $A$ of dimension $n$, defined over a number field $K$, is said to have \textit{complex multiplication}, or to be a \textit{CM abelian variety}, if its endomorphism algebra $\operatorname{End}^0(A) = \operatorname{End}(A) \otimes \Q$ contains a commutative subring of dimension $2n$ over $\Q$ (see e.g. \cite[Proposition 3.3]{Mil20}). Moreover, let $F$ be a totally real number field of degree $n$ and let $E$ be a CM field with maximal totally real subfield $F$ and let $\Phi$ a CM type for $E$. Then, we say that an abelian variety $A$ is of type $(\mathcal{O}_E,\Phi)$ if it admits complex multiplication by $\mathcal{O}_E$, in the sense that $[E : \Q] = 2 \dim{(A)}$ and there is a ring homomorphism $\mathcal{O}_E \hooklongrightarrow \operatorname{End}(A)$.

Thus let $X_{\Phi}$ is said to be a CM abelian variety defined over $\overline{\Q}$ of type $(\mathcal{O}_E,\Phi)$. Also, let $L \subseteq \overline{\mathbb Q}$ be a number field over which $X_{\Phi}$ has everywhere good reduction 
and choose a differential $\omega \in H^{0}(X_{\Phi}, \Omega^n_{X_{\Phi}})$. Then the \textit{Faltings height} of $X_{\Phi}$ is defined by
\begin{align*}
h_{\textrm{Fal}}(X_{\Phi}):= -\frac{1}{2[L:\mathbb Q]}\sum_{\sigma : L \hookrightarrow \mathbb{C}}
\log\left| \int_{X_{\Phi}^{\sigma}(\mathbb{C})} \omega^{\sigma} \wedge \overline{\omega^{\sigma}}\right|.
\end{align*}
It is known that the Faltings height $h_{\textrm{Fal}}(X_{\Phi})$ does not depend on the choice of $L, \omega$, or $X_{\Phi}$. In particular, by a result of Colmez \cite[Théor\`{e}me 0.3 (ii)]{Col93}, the Faltings height depends only on the choice of CM type $\Phi$ and not on the choice of CM abelian variety $X_{\Phi}$. This means that if $X_{\Phi}$ and $Y_{\Phi}$ are CM abelian varieties of type $(\mathcal{O}_E, \Phi)$, then $h_{\textrm{Fal}}(X_{\Phi})=h_{\textrm{Fal}}(Y_{\Phi})$. Hence it is often denoted just by $h_{\textrm{Fal}}(\Phi)$. We observe that in the literature there are different normalizations of the Faltings height and so some versions will differ by the addition of a constant.

Let $\Q^{\CM}$ be the compositum of all CM fields and $G^{\CM}:= \Gal(\Q^{\CM} / \Q)$. Given an irreducible Artin character $\chi_{\pi}$ of $G^{\CM}$ we let $L(\chi_{\pi},s)$ be the (incomplete) Artin $L$--function of $\chi_{\pi}$ and $\mathfrak{f}_{\chi_{\pi}}$ be the analytic Artin conductor of $\chi_{\pi}$. Then, Colmez \cite[Conjecture 0.4]{Col93} made the following conjecture (see also \cite{Yan10}).

\begin{conjecture}\label{Colmez-conjecture} 
Let $E$ be a CM field. Then if $\Phi$ is any CM type for $E$ and
$X_{\Phi}$ any CM abelian variety of type $(\mathcal{O}_E, \Phi)$, we have
\begin{align*}
h_{\mathrm{Fal}}(X_{\Phi}) = \sum_{\chi_{\pi}} c_{\Phi, \chi_{\pi}} \left( -\frac{L'(\chi_{\pi}, 0)}{L(\chi_{\pi}, 0)} - \frac{1}{2} \log{(\mathfrak{f}_{\chi_{\pi}})} \right),
\end{align*}
where the sum is over all the irreducible Artin characters of $G^{\CM}$, the coefficients $c_{\Phi, \chi_{\pi}}$ can be explicitly described and only a finite number of them are nonzero.
\end{conjecture}

Originally, Colmez \cite[Théor\`{e}me 0.5]{Col93} was able to prove that his conjecture is true for any \textit{abelian} CM field $E$, up to an unknown rational multiple of $\log{2}$, which was removed in later work by A. Obus \cite{Obu13}. Thus, the conjecture only remains open in general for \textit{non-abelian} CM fields. Interestingly, in this direction, and certainly sparked by the very important work of T. Yang \cite{Yan10} and \cite{Yan13}, the last decade has seen a great amount of work on the \textit{non-abelian} case of this conjecture. In particular, the biggest advance was a proof of an averaged version of the conjecture by S. Yuan and S. Zhang \cite{YZ18} and independently by F. Andreatta, E. Goren, B. Howard and K. Madapusi-Pera \cite{AGHM18}.

Now, given a CM field $E$ of degree $2n$ with maximal totally real subfield $F$, we let $\Phi(E)$ be the set of CM types of $E$. It turns out that the Galois group $\Gal(E^c/\Q)$ (or equivalently $G^{\mathcal{CM}}$) acts on $\Phi(E)$ by composition of maps (see e.g \cite[Definition 2.1]{BSM18}). Then, using the proved averaged version of the Colmez conjecture as the basis, the first author and R. Masri showed that if this action is transitive, then Conjecture \ref{Colmez-conjecture} is true for the CM field $E$ (see \cite[Proposition 5.1]{BSM18}). In that same paper, using this result as a stepping stone, an approach to study the Colmez conjecture from an arithmetic statistics perspective was initiated (see \cite[\S 1.2]{BSM18}). Later on, that approach was carried out by the first author, R. Masri and F. Thorne in \cite{BSMT17}.

More precisely, in the latter paper, the authors defined the notion of a CM field with maximal Galois group with regards to an embedding into a wreath product. Specifically, according to Theorem \ref{Barquero-Sanchez-Masri-Thorne-Theorem}, if $G \leq S_n$ is a transitive subgroup with $\Gal(F^c/\Q) \simeq G$, then the Galois group $\Gal(E^c/\Q)$ is isomorphic to a subgroup of the wreath product $C_2 \wr_{\Omega} G$, where $\Omega = \{ 1, 2, \dots, n \}$ and the action of $G$ on $C_2^{\Omega} = C_2^n$ is by permuting the components. Thus the CM field $E$ is said to be a \textit{$G$-Weyl CM field} if $\Gal(E^c/\Q)$ is maximal, i.e., if $\Gal(E^c/\Q) \simeq C_2 \wr_{\Omega} G$. Using this notion, the authors proved in \cite[Theorem 1.14]{BSMT17} that if $E$ is a $G$-Weyl CM field, then the Colmez conjecture \ref{Colmez-conjecture} is true for $E$. 

Now, for any $X>0$ let $N_{2n}^{\text{Weyl}}(X, G)$ count the number of $G$-Weyl CM fields $E$ of degree $2n$ with discriminant $|d_E| \leq X$ and let 
$$
N_{2n}^{\text{Weyl}}(X) := \sum_{G \leq S_n} N_{2n}^{\text{Weyl}}(X, G).
$$
Hence $N_{2n}^{\text{Weyl}}(X)$ counts the total number of $G$-Weyl CM fields of degree $2n$ with $|d_E| \leq X$ for all the possible transitive subgroups $G$ of $S_n$. Moreover, we let $N_{2n}^{\text{cm}}(X)$ count the total number of CM fields $E$ of degree $2n$ with $|d_E| \leq X$. 

Then, assuming a weak form of the upper bound in the Malle conjecture (\ref{malle-conjecture-1}), which is known to be true in many cases (see e.g \cite[p. 5, Remark 1.8 and \S 4]{BSMT17}), it was proved in that paper \cite[Theorem 1.11]{BSMT17} that there exist positive constants $C(n, G)$, $\alpha(n, G)$ and $\beta(n)$ such that
$$
\frac{N_{2n}^{\text{Weyl}}(X, G)}{N_{2n}^{\text{cm}}(X)} = C(n, G) + O(X^{-\alpha(n, G)})
$$
and
$$
\frac{N_{2n}^{\text{Weyl}}(X)}{N_{2n}^{\text{cm}}(X)} = 1 + O(X^{-\beta(n)}).
$$
This last result shows that for a fixed $n\geq 1$, out of all the possible CM fields of degree $2n$, it happens that 100\% of them are of Weyl type. Therefore, when combined with the aforementioned results about the Colmez conjecture, this asymptotics imply that, conditional on the weak version of the upper bound in the Malle conjecture \eqref{malle-conjecture-1}, the Colmez conjecture holds true for 100\% of CM fields of degree $2n$ for any $n \geq 1$. For more details about this the reader is referred to \cite[\S 1.2]{BSMT17}.

\subsection{Application in arithmetic geometry}

This subsection describes the application that is by its nature the more demanding of the reader in terms of background. Since we only want to provide an indication of the usefulness of the embedding theorems of the paper, we won't go into the full details of the definitions of the concepts mentioned here and instead we refer the interested reader to the original paper of C-L. Chai and F. Oort \cite{CO12}, and to the excellent notes of J. S. Milne on Complex Multiplication \cite{Mil20} and on Shimura Varieties \cite{Mil17}.

To begin with, in \cite{CO12}, the authors study the question of \textit{whether there exists an abelian variety $A$, defined over the field $\overline{\Q}$ of all algebraic numbers, that is not isogenous to the Jacobian of a stable algebraic curve over $\overline{\Q}$}.

Now, let $\mathcal{A}_{g, 1}$ be the moduli space of principally polarized abelian varieties. A subset of $\mathcal{A}_{g, 1}$ is called a \textit{special subset} if it is a finite union of \textit{Shimura varieties}. Also, for a point $x = [(A, \lambda)] \in \mathcal{A}_{g, 1}(\overline{\Q})$ with $A$ an abelian variety with a principal polarization $\lambda$ over $\overline{\Q}$, the \textit{isogeny orbit} $\mathcal{I}(x)$ of $x$ in $\mathcal{A}_{g, 1}$ is defined to be the set of all points $y = [(B, \nu)] \in \mathcal{A}_{g, 1}(\overline{\Q})$ such that the abelian variety $B$ is isogenous to $A$.

In their paper, Chai and Oort reduce their problem to proving that \textit{for any special subset $Y \subsetneq \mathcal{A}_{g, 1}$ over $\overline{\Q}$, there exists a CM-point $y \in \mathcal{A}_{g, 1}(\overline{\Q})$ such that the isogeny orbit $\mathcal{I}(y)$ of $y$ and the special subset $Y$ are disjoint}.

In order to do this, they define a notion of CM-points of sufficient generality, which, as they explain, are expected to have density one among all the CM-points in $\mathcal{A}_{g, 1}$ under any reasonable counting procedure (see \cite[Footnote 2]{CO12}).

Then, for the sake of defining those special CM-points, they first prove Theorem \ref{Chai-Oort-Theorem} (see \cite[Lemma 2.8]{CO12}), that is, that if $K$ is a number field of degree $g = [K : \Q]$ and $L$ is a quadratic extension of $K$, then the Galois group $\Gal(L^c/\Q)$ embeds as a subgroup of the wreath product $C_2 \wr_{\Omega} S_g$, where $\Omega = \{ 1, \dots, g \}$. Using this, they define a \textit{Weyl CM field} of degree $2g$ as a CM field $L$ with maximal totally real subfield $K$ of degree $g$, such that the Galois group $\Gal(L^c/K)$ is maximal, i.e., such that $\Gal(L^c/\Q) \simeq C_2 \wr_{\Omega} S_g$. Moreover, with this, Chai and Oort define a \textit{Weyl CM point} to be a point $[(A, \lambda)] \in \mathcal{A}_{g, 1}(\overline{\Q})$ satisfying that the endomorphism algebra $\operatorname{End}^0(A) = L$, where $L$ is a Weyl CM field of degree $2g$ (see \cite[Definition 2.11]{CO12}).

Finally, as Chai and Oort explain, they use the notion of a Weyl CM point to prove that an irreducible Shimura subvariety $S \subsetneq \mathcal{A}_{g, 1}$ of positive dimension that contains a Weyl CM point $[(A, \lambda)]$ with endomorphism algebra $\operatorname{End}^{0}(A) = L$ as above, is actually a Hilbert modular variety associated to the maximal totally real subfield of $L$ (see \cite[p. 3 and Lemma 3.5]{CO12}).

\section{Wreath products of groups}\label{Wreath-section}

In this section we describe the wreath product construction on groups. This is defined in terms of semidirect products so we begin by recalling the definition of a semidirect product.

\subsection{Semidirect products}

Let $K$ and $H$ be groups and let $\theta \colon H \longrightarrow \op{Aut}(K)$ be a group homomorphism, where for every $h \in H$ we denote $\theta(h)$ by $\theta_h$ in order to simplify the notation. Then, then \textit{(outer) semidirect product} of $K$ and $H$ with respect to $\theta$, denoted $K \rtimes_{\theta} H$, is the group whose underlying set is the cartesian product $H \times K$, with group operation given by 
$$
(k_1, h_1)(k_2, h_2) := (k_1 \theta_{h_1}(k_2), h_1 h_2)
$$
for arbitrary $k_1, k_2 \in K$ and $h_1, h_2 \in H$. The identity element in $K \rtimes_{\theta} H$ is $(e_K, e_H)$ and the inverse of an element $(k, h) \in K \rtimes_{\theta} H$ is given by $(k, h)^{-1} = (\theta_{h^{-1}}(k^{-1}), h^{-1})$.

\subsection{Wreath products}

Let $\Omega$ be an arbitrary set and let $K$ and $H$ be groups. The set $K^{\Omega}$ of all functions $f \colon \Omega \to K$ forms a group under the binary operation of pointwise multiplication of functions, i.e., if $f, g \in K^{\Omega}$, then their product is the function $fg \colon \Omega \to K$ defined by $(fg)(\omega) = f(\omega) g(\omega)$ for every $\omega \in \Omega$.

Suppose that $H$ acts on the left on $\Omega$, i.e., that there's an action $\cdot \colon H \times \Omega \to \Omega$ such that
\begin{itemize}
	\item[(i)] $h_1 \cdot (h_2 \cdot \omega) = (h_1h_2) \cdot \omega$ for every $h_1, h_2 \in H$ and every $\omega \in \Omega$; and
	\item[(ii)] $e_H \cdot \omega = \omega$ for every $\omega \in \Omega$.
\end{itemize}

From this action we define a homomorphism in the next proposition.

\begin{proposition}\label{theta-homomorphism}
There is a homomorphism of groups
\begin{align*}
	\theta \colon &H \longrightarrow \op{Aut}(K^{\Omega})\\
	& h \longmapsto \theta_h
\end{align*}
where $\theta_h \colon K^{\Omega} \to K^{\Omega}$ is an automorphism, defined by $\theta_h(f)(\omega) := f(h^{-1}\cdot \omega)$ for every $f \in K^{\Omega}$ and every $\omega \in \Omega$. 
\end{proposition}

\begin{proof}
First, we show that $\theta_h \in \op{Aut}(K^{\Omega})$. Let $f, g \in K^{\Omega}$. Then certainly $\theta_h(f) \in K^{\Omega}$ by definition. Moreover, $\theta_h$ is a group homomorphism because if $\omega \in \Omega$, then 
\begin{align*}
	\theta_h(fg)(\omega) = (fg)(h^{-1}\cdot \omega) = f(h^{-1} \cdot \omega) g(h^{-1}\cdot \omega) = \big(\theta_h(f)(\omega) \big) \big( \theta_h(g)(\omega) \big).
\end{align*}
Thus $\theta_h(fg) = \theta_h(f) \theta_h(g)$ so $\theta_h$ is a homomorphism. Now, we want to check that $\theta_h$ is a bijective function.

\noindent \textit{Injectivity of $\theta_h$:} Suppose that $\theta_h(f) = \theta_h(g)$ and let $\omega \in \Omega$. Then note that this implies that
\begin{align*}
	&\quad \theta_h(f)(h \cdot \omega) = \theta_h(g)(h \cdot \omega)\\
	\implies & f(h^{-1}\cdot(h\cdot \omega)) = g(h^{-1} \cdot (h \cdot \omega))\\
	\implies &  f((h^{-1} h )\cdot \omega) = g((h^{-1}h) \cdot \omega)\\
	\implies & f(e_H \cdot \omega) = g(e_H \cdot \omega)\\
	\implies & f(\omega) = g(\omega).
\end{align*}
Thus we conclude that $f = g$, so that $\theta_h$ is injective.

\noindent \textit{Surjectivity of $\theta_h$:} Let $g \in K^{\Omega}$. Define a function $f \colon \Omega \to K$ by $f(\omega) := g(h \cdot \omega)$ for every $\omega \in \Omega$. Then observe that $\theta_h(f)(\omega) = f(h^{-1}\cdot \omega) = g(h \cdot (h^{-1} \cdot \omega)) = g((h \cdot h^{-1}) \cdot \omega) = g(e_H \cdot \omega) = g(\omega)$. In other words, this shows that $\theta_h(f) = g$, so $\theta_h$ is surjective.

This completes the proof that $\theta_h \in \op{Aut}(K^{\Omega})$ for every $h \in H$. 

\noindent \textit{$\theta$ is a group homomorphism:} Now we must show that the map $\theta \colon H \to \op{Aut}(K^{\Omega})$ is a group homomorphism. Thus let $h_1, h_2 \in H$. Then we want to prove that $\theta(h_1 h_2) = \theta(h_1) \circ \theta(h_2)$, that is, that $\theta_{h_1h_2} = \theta_{h_1} \circ \theta_{h_2}$, where $\circ$ denotes the operation of composition of functions, which is the binary operation in $\op{Aut}(K^{\Omega})$. To see this, again let $f \in K^{\Omega}$ and let $\omega \in \Omega$. Then 
\begin{align*}
	\theta_{h_1h_2}(f)(\omega) &= f((h_1h_2)^{-1} \cdot \omega) \\
	&= f((h_2^{-1}h_1^{-1}) \cdot \omega) \\
	&= f(h_2^{-1} \cdot (h_1^{-1} \cdot \omega)) \\
	&= \theta_{h_2}(f)(h_1^{-1}\cdot \omega) \\
	&= \theta_{h_1} \big( \theta_{h_2}(f) \big)(\omega) \\
	&= \big( \theta_{h_1} \circ \theta_{h_2} \big)(f)(\omega).
\end{align*}
This proves that $\theta_{h_1 h_2}(f) = (\theta_{h_1} \circ \theta_{h_2})(f)$ and hence this implies that $\theta(h_1 h_2) = \theta(h_1) \circ \theta(h_2)$, as we wanted to show. Thus $\theta$ is a group homomorphism and this completes the proof of the proposition.

\end{proof}

\begin{remark}
	The reader might wonder why in defining the map $\theta \colon H \to \op{Aut}(K^{\Omega})$ in the previous proposition, we defined $\theta_h(f)(\omega)$ as $f(h^{-1} \cdot \omega)$ instead of defining it as $f(h \cdot \omega)$? The reason is that if it had been defined in the latter way, then we wouldn't have been able to prove that $\theta$ is a group homomorphism, but instead an antihomomorphism. More explicitly, if we had defined $\theta_h(f)(\omega) = f(h \cdot \omega)$ for $h \in H$, $f \in K^{\Omega}$ and $\omega \in \Omega$, then for $h_1, h_2 \in H$ we would have obtained
\begin{align*}
	\theta_{h_1h_2}(f)(\omega) &= f((h_1h_2) \cdot \omega) \\
	&= f(h_1 \cdot (h_2 \cdot \omega)) \\
	&= \theta_{h_1}(f)(h_2 \cdot \omega) \\
	&= \theta_{h_2} \big( \theta_{h_1}(f) \big)(\omega) \\
	&= \big( \theta_{h_2} \circ \theta_{h_1} \big)(f)(\omega).
\end{align*}
In other words, with that definition $\theta$ would have been an antihomomorphism, that is, a map satisfying the relation $\theta(h_1h_2) = \theta_{h_1h_2} = \theta_{h_2} \circ \theta_{h_1} = \theta(h_2) \circ \theta(h_1)$.
\end{remark}

\begin{definition}
Let $K$ and $H$ be groups, let $\Omega$ be an $H$-set, that is, a set equipped with a left action of $H$ given by $\cdot \colon H \times \Omega \longrightarrow \Omega$, and let $\theta \colon H \longrightarrow \op{Aut}(K^{\Omega})$ be the homomorphism defined in Proposition \ref{theta-homomorphism}. Then the \textit{(complete) wreath product} of the groups $K$ and $H$ with respect to the action of $H$ on $\Omega$, denoted $K \wr_{\Omega} H$, is defined by the semidirect product
$$
K \wr_{\Omega} H := K^{\Omega} \rtimes_{\theta} H.
$$
The group $K^{\Omega}$ is called the \textit{base} of the wreath product. Thus, the binary operation on the wreath product $K \wr_{\Omega} H$ is given by 
\begin{align}\label{wreath-multiplication}
	(f_1, h_1)(f_2, h_2) = (f_1 \theta_{h_1}(f_2), h_1h_2)
\end{align}
for $f_1, f_2 \in K^{\Omega}$ and $h_1, h_2 \in H$. Note that the product $f_1 \theta_{h_1}(f_2)$ is the product of functions in $K^{\Omega}$, which was defined pointwise.

In the particular case in which the set $\Omega = H$ and the action of $H$ on $\Omega$ is the left regular action, that is, when $H$ acts on the left on itself by left multiplication, then the corresponding wreath product is called the \textit{regular wreath product} of $K$ and $H$ and it is denoted by $K \wr_r H$.  
\end{definition}

\begin{remark}\label{Wreath-size}
	When $\Omega$ is finite, the order of the wreath product is given by $|K \wr_{\Omega} H| = |K|^{|\Omega|} \cdot |H|$.
\end{remark}

\begin{remark}
	What we have called the \textit{regular wreath product} is also called the \textit{standard complete wreath product} by some authors (see e.g. \cite[p. 326]{Rob96}). The terminology \textit{regular wreath product} is used for example by Rotman \cite[p. 175]{Rot95}.
\end{remark}

\begin{remark}\label{tuple-remark}
	If the functions in $K^{\Omega}$ are written in ``tuple notation", that is, if we write a function $f \colon \Omega \to K$ as a sequence $(a_{\omega})_{\omega \in \Omega}$, where $f(\omega) = a_{\omega} \in K$ for every $\omega \in \Omega$, then the relation $\theta_h(f)(\omega) = f(h^{-1} \cdot \omega)$ defining the automorphism $\theta_h \in \op{Aut}(K^{\Omega})$ is written as
	$$
	\theta_h((a_{\omega})_{\omega \in \Omega}) = (a_{h^{-1}\cdot \omega})_{\omega \in \Omega}.
	$$
This means that the $\omega$-th component of $\theta_h((a_{\omega})_{\omega \in \Omega})$ is equal to the $(h^{-1} \cdot \omega)$-th component of $(a_\omega)_{\omega \in \Omega}$.
\end{remark}

We now give a detailed example of a wreath product in order to consider these definitions in a concrete case. This example will be reconsidered later on as part of Example \ref{Biquadratic-example}.

\begin{example}\label{example-c2-regular-c2}
	Let $C_2 = \Z/2\Z = \{ [0]_2, [1]_2 \}$. We will now consider the regular wreath product $C_2 \wr_r C_2 = C_2^{C_2} \rtimes_{\theta} C_2$. This group has order $|C_2|^{|C_2|} \cdot |C_2| = 2^2 \cdot 2 = 8$. We note that here the operation on the group $C_2^{C_2}$ is given by pointwise addition, i.e., if $f, g \in C_2^{C_2}$, then $(f + g)([n]_2) = f([n]_1) + g([n]_2)$ for every $[n]_2 \in C_2$. Thus, in this example, certain group operations will use additive notation.
	
	Now, the group $C_2^{C_2} = \{ f_1, f_2, f_3, f_4 \}$ consists of the four functions
	\begin{align*}
		f_1 &\colon [0]_2 \mapsto [0]_2, \quad [1]_2 \mapsto [0]_2,\\
		f_2 &\colon [0]_2 \mapsto [0]_2, \quad [1]_2 \mapsto [1]_2,\\
		f_3 &\colon [0]_2 \mapsto [1]_2, \quad [1]_2 \mapsto [0]_2,\\
		f_4 &\colon [0]_2 \mapsto [1]_2, \quad [1]_2 \mapsto [1]_2.
	\end{align*}
	Here the function $f_1$ is the identity elelement in $C_2^{C_2}$ and moreover we have $C_2^{C_2} \simeq C_2 \oplus C_2$. Thus, as a set, the regular wreath product $C_2 \wr_r C_2$ is given by
	$$
	C_2 \wr_r C_2 = \{ (f_1, [0]_2), (f_2, [0]_2), (f_3, [0]_2), (f_4, [0]_2), (f_1, [1]_2), (f_2, [1]_2), (f_3, [1]_2), (f_4, [1]_2) \}.
	$$
	As we showed, this is a group of order 8. Up to isomorphism, we know that there are five groups of order 8, namely, the three abelian ones $C_8$, $C_4 \oplus C_2$ and $C_2 \oplus C_2 \oplus C_2$, and the two non-abelian ones given by the dihedral group $D_4$ and the quaternion group $Q_8$.
	
	We will now show that $C_2 \wr_r C_2 \simeq D_4$. To see this, we will prove that $C_2 \wr_r C_2$ satisfies the standard presentation for the dihedral group $D_4$, that is, the presentation $\langle x, y \mid x^4 = y^2 = e, \, yxy^{-1} = x^{-1}   \rangle$ (see e.g. \cite[p. 67]{Hun80}).
	
	Now, the map $\theta \colon C_2 \longrightarrow \op{Aut}(C_2^{C_2})$ is determined by the two automorphisms
	$$
	\theta_{[0]_2} \colon C_2^{C_2} \longrightarrow C_2^{C_2} \quad \text{and} \quad \theta_{[1]_2} \colon C_2^{C_2} \longrightarrow C_2^{C_2}
	$$
	that we compute below. 
	
	For $\theta_{[0]_2}$, note that given a function $f \colon C_2 \longrightarrow C_2$, we have
	$$
	\theta_{[0]_2}(f)([n]_2) = f(-[0]_2 + [n]_2) = f([n]_2)
	$$
	and hence $\theta_{[0]_2}(f) = f$ for every function $f \in C_2^{C_2}$.
	
	Similarly, for $\theta_{[1]_2}$ we have
	$$
	\theta_{[1]_2}(f)([n]_2) = f(-[1]_2 + [n]_2) = f([n + 1]_2).
	$$
	From this we see that
	\begin{align*}
		\theta_{[1]_2}(f_1) \colon [0]_2 \mapsto f_1([1]_2) = [0]_2, \quad [1]_2 \mapsto f_1([0]_2) = [0]_2,\\
		\theta_{[1]_2}(f_2) \colon [0]_2 \mapsto f_2([1]_2) = [1]_2, \quad [1]_2 \mapsto f_2([0]_2) = [0]_2,\\
		\theta_{[1]_2}(f_3) \colon [0]_2 \mapsto f_3([1]_2) = [0]_2, \quad [1]_2 \mapsto f_3([0]_2) = [1]_2,\\
		\theta_{[1]_2}(f_4) \colon [0]_2 \mapsto f_4([1]_2) = [1]_2, \quad [1]_2 \mapsto f_4([0]_2) = [1]_2.
	\end{align*}
	Thus $\theta_{[1]_2}(f_1) = f_1$, $\theta_{[1]_2}(f_2) = f_3$, $\theta_{[1]_2}(f_3) = f_2$ and $\theta_{[1]_2}(f_4) = f_4$.
	
	With this information we can compute the product of any two elements in the wreath product $C_2 \wr_r C_2$. We note that the element $(f_1, [0]_2)$ is the identity element in $C_2 \wr_r C_2$.
	
	Now let $x := (f_2, [1]_2) \in C_2 \wr_r C_2$. We will check that $x$ has order 4. Observe for instance that
	\begin{align*}
		x^2 = (f_2, [1]_2) (f_2, [1]_2) = (f_2 + \theta_{[1]_2}(f_2), [1]_2 + [1]_2)) = (f_2 + f_3, [0]_2) = (f_4, [0]_2).
	\end{align*}
	By doing similar calculations one can check that $x^3 = (f_3, [1]_2)$ and $x^4 = (f_1, [0]_2)$. 
	
	In a similar way, let $y := (f_1, [1]_2) \in C_2 \wr_r C_2$. Then one can check that $y$ has order 2 and hence we have $x^4 = y^2 = e$.
	
	Now, since $x$ has order 4, we have $x^{-1} = x^3$, that is $x^{-1} = (f_3, [1]_2)$. Similarly, since $y$ has order 2, we have $y^{-1} = y$. Then we can check that
	\begin{align*}
		yxy^{-1} &= (f_1, [1]_2)(f_2, [1]_2)(f_1, [1]_2)\\
		&= (f_1 + \theta_{[1]_2}(f_2), [0]_2)(f_1, [1]_2)\\
		&= (f_1 + f_3, [0]_2)(f_1, [1]_2)\\
		&= (f_3, [0]_2) (f_1, [1]_2)\\
		&= (f_3 + \theta_{[0]_2}(f_1), [1]_2)\\
		&= (f_3 + f_1, [1]_2)\\
		&= (f_3, [1]_2)\\
		&= x^{-1}.
	\end{align*}
	Therefore, these calculations imply that $C_2 \wr_r C_2$ has the presentation
	$$
	C_2 \wr_r C_2 = \langle x, y \mid x^4 = y^2 = e, \, yxy^{-1} = x^{-1} \rangle
	$$
	and from this we conclude that $C_2 \wr_r C_2 \simeq D_4$.
\end{example}

\section{Some technical results about isomorphisms and subgroups of wreath products}\label{Wreath-properties-section}\label{Technical-section}

As the title indicates, in this section we prove some technical results about isomorphisms and subgroups of wreath products. These results will be useful in identifying the wreath products into which the Galois groups are being embedded in the main theorems of the paper, as will be illustrated in the examples from Section \ref{Examples-section}.

The first result we prove allows for the identification of complete wreath products with isomorphic factors.

\begin{proposition}\label{Omega-wreath-isomorphism}
Let $K, \widehat{K}, H, \widehat{H}$ be arbitrary groups with $K \simeq \widehat{K}$ and $H \simeq \widehat{H}$, and let $\varphi \colon H \longrightarrow\widehat{H}$ and $\psi \colon K \longrightarrow \widehat{K}$ be isomorphisms. Suppose that $\Omega$ is an $H$-set, that $\widehat{\Omega}$ is an $\widehat{H}$-set, and that there is an equivariant bijection $\xi: \Omega \longrightarrow \widehat{\Omega}$ such that $\xi(h\cdot \omega) = \varphi(h) \cdot \xi(\omega)$ for all $\omega \in \Omega$ and $h\in H$. Then $K \wr_{\Omega} H \simeq \widehat{K} \wr_{\widehat{\Omega}} \widehat{H}$ and an isomorphism is given by 
	\begin{align*}
		\vartheta \colon &K \wr_{\Omega} H \longrightarrow \widehat{K} \wr_{\widehat{\Omega}} \widehat{H}\\
		&\   (f, h) \longmapsto \big( \psi \circ f \circ \xi^{-1}, \varphi(h) \big).
	\end{align*}
\end{proposition}

\begin{proof}
	The direct proof that $\vartheta$ is injective and surjective is straightforward, so we will just indicate that the inverse of $\vartheta$ is given by 
	\begin{align*}
		\vartheta^{-1} \colon &\widehat{K} \wr_{\widehat{\Omega}} \widehat{H}  \longrightarrow K \wr_{\Omega} H\\
		&\   (\hat{f}, \hat{h}) \longmapsto \big( \psi^{-1} \circ \hat{f} \circ \xi, \varphi^{-1}(\hat{h}) \big).
	\end{align*}
	Therefore, we must prove that $\vartheta$ is a group homomorphism. First we let $\theta \colon H \longrightarrow \op{Aut}(K^{\Omega})$ and $\widehat{\theta} \colon \widehat{H} \longrightarrow \op{Aut}(\widehat{K}^{\widehat{\Omega}})$ be the homomorphisms from Proposition \ref{theta-homomorphism}, which for $h \in H$, $\hat{h} \in \widehat{H}$, $f \colon \Omega \longrightarrow K$ and $\hat{f} \colon \widehat{\Omega} \longrightarrow \widehat{K}$ are defined by
	$$
	\theta_h(f)(\omega_0) = f(h^{-1} \cdot \omega_0) \quad \text{and} \quad \widehat{\theta}(\hat{f})(\hat{\omega}_0) = \hat{f}(\hat{h}^{-1} \cdot \hat{\omega}_0)
	$$
	for every $\omega_0 \in H$ and every $\hat{\omega}_0 \in \widehat{H}$.
	
	Now let $(f_1, h_1), (f_2, h_2) \in K \wr_{\Omega} H$. We will prove that 
	$$
	\vartheta \big( (f_1, h_1) (f_2, h_2) \big) = \vartheta \big( (f_1, h_1) \big) \vartheta \big( (f_2, h_2) \big).
	$$
	Observe first that
	\begin{align}\label{vartheta-homo-1}
		\vartheta \big( (f_1, h_1) (f_2, h_2) \big) &= \vartheta \big( (f_1 \theta_{h_1}(f_2), h_1 h_2) \big) \notag\\
		&= \big( \psi \circ ( f_1 \theta_{h_1}(f_2) ) \circ \xi^{-1}, \varphi(h_1 h_2) \big) \notag\\
		&= \big( \psi \circ ( f_1 \theta_{h_1}(f_2) ) \circ \xi^{-1}, \varphi(h_1) \varphi(h_2) \big).
	\end{align}
	On the other hand
	\begin{align}\label{vartheta-homo-2}
		\vartheta \big( (f_1, h_1) \big) \vartheta \big( (f_2, h_2) \big) &= \big( \psi \circ f_1 \circ \xi^{-1}, \varphi(h_1) \big)  \big( \psi \circ f_2 \circ \xi^{-1}, \varphi(h_2) \big) \notag\\
		&= \big( (\psi \circ f_1 \circ \xi^{-1}) \widehat{\theta}_{\varphi(h_1)}(\psi \circ f_2 \circ \xi^{-1}), \varphi(h_1) \varphi(h_2)  \big).
	\end{align}
	Therefore, comparing equations (\ref{vartheta-homo-1}) and (\ref{vartheta-homo-2}), we must prove that
	\begin{align}\label{vartheta-homo-3}
		\psi \circ ( f_1 \theta_{h_1}(f_2) ) \circ \xi^{-1} = (\psi \circ f_1 \circ \xi^{-1}) \widehat{\theta}_{\varphi(h_1)}(\psi \circ f_2 \circ \xi^{-1}).
	\end{align}
	
	Note that both the function on the left hand side and the one on the right hand side of (\ref{vartheta-homo-3}) are functions in $\widehat{K}^{\widehat{\Omega}}$. Hence, we will check that both sides give the same value when evaluated at an arbitrary element $\hat{\omega} \in \widehat{\Omega}$. 
	
	For the LHS of equation (\ref{vartheta-homo-3}) we have
	\begin{align}\label{LHS-vartheta-homo}
		\big( \psi \circ ( f_1 \theta_{h_1}(f_2) ) \circ \xi^{-1} \big) (\hat{\omega}) &= \psi \big( (f_1 \theta_{h_1}(f_2))(\xi^{-1}(\hat{\omega})) \big) \notag\\
		&= \psi \big( f_1(\xi^{-1}(\hat{\omega})) \theta_{h_1}(f_2)(\xi^{-1}(\hat{\omega})) \big) \notag\\
		&= \psi \big( f_1(\xi^{-1}(\hat{\omega})) f_2(h_1^{-1} \cdot \xi^{-1}(\hat{\omega}) \big) \notag\\
		&= \psi \big( f_1(\xi^{-1}(\hat{\omega})) \big) \psi \big( f_2(h_1^{-1} \cdot \xi^{-1}(\hat{\omega})) \big).
	\end{align}
	On the other hand, for the RHS of equation (\ref{vartheta-homo-3}) we have
	\begin{align}\label{RHS-vartheta-homo}
		\big( (\psi \circ f_1 \circ \xi^{-1}) \widehat{\theta}_{\varphi(h_1)}(\psi \circ f_2 \circ \xi^{-1}) \big) (\hat{\omega}) &= (\psi \circ f_1 \circ \xi^{-1}) (\hat{\omega}) \,   \widehat{\theta}_{\varphi(h_1)}(\psi \circ f_2 \circ \xi^{-1})(\hat{\omega}) \notag\\
		&= \psi \big( f_1(\xi^{-1}(\hat{\omega})) \big)  \big( \psi \circ f_2 \circ \xi^{-1} \big) (\varphi(h_1)^{-1} \cdot \hat{\omega}) \notag \\
		&= \psi \big( f_1(\xi^{-1}(\hat{\omega})) \big) \psi \big( f_2(\xi^{-1}(\varphi(h_1)^{-1} \cdot \hat{\omega})) \big) \notag\\
		&= \psi \big( f_1(\xi^{-1}(\hat{\omega})) \big) \psi \big( f_2(\xi^{-1}(\varphi(h_1^{-1}) \cdot \hat{\omega})) \big).
	\end{align}
        Now, observe that as a consequence of the equivariance of the map $\xi$ we have
        \begin{align}\label{xi-condition}
        h_1^{-1} \cdot \xi^{-1}(\hat{\omega}) = \xi^{-1}(\varphi(h_1^{-1}) \cdot \hat{\omega}).
        \end{align}
        Finally, comparing (\ref{LHS-vartheta-homo}) with (\ref{RHS-vartheta-homo}) and using (\ref{xi-condition}), we see that (\ref{vartheta-homo-3}) holds, and that implies that $\vartheta$ is a homomorphism. This completes the proof of the proposition.
\end{proof}

Now, taking $\Omega = H$, $\widehat{\Omega} = \widehat{H}$ and $\xi = \varphi$ in the previous proposition, we get the corresponding result for regular wreath products, which is stated in the following corollary. In this case the condition $\xi(h\cdot \omega) = \varphi(h) \cdot \xi(\omega)$ is reduced to the fact that $\varphi$ is a group homomorphism, because the action of $H$ on $\Omega=H$ is the left regular action, i.e. $H$ acts on itself on the left by left multiplication.

\begin{corollary}\label{Regular-wreath-isomorphism}
    Let $K, \widehat{K}, H, \widehat{H}$ be arbitrary groups with $K \simeq \widehat{K}$ and $H \simeq \widehat{H}$, and let $\varphi \colon H \longrightarrow\widehat{H}$ and $\psi \colon K \longrightarrow \widehat{K}$ be isomorphisms. Then $K \wr_r H \simeq \widehat{K} \wr_r \widehat{H}$ and an isomorphism is given by 
	\begin{align*}
		\vartheta \colon &K \wr_r H \longrightarrow \widehat{K} \wr_r \widehat{H}\\
		&\   (f, h) \longmapsto \big( \psi \circ f \circ \varphi^{-1}, \varphi(h) \big).
	\end{align*}
\end{corollary}

\begin{remark}
    We note that Proposition \ref{Omega-wreath-isomorphism} can be applied in conjunction with Theorem \ref{Main-Theorem-2}, Theorem \ref{Kummer-embedding} and Corollaries \ref{Cyclic-embedding} and \ref{Quadratic-embedding}. On the other hand, Corollary \ref{Regular-wreath-isomorphism} can be applied in conjunction with Theorem \ref{Main-Theorem}.
\end{remark}

Finally, the following result gives conditions under which a wreath product formed from subgroups of the original components is actually a subgroup of the original wreath product.

\begin{corollary}\label{Subgroup-wreath-isomorphism}
    Let $\widehat{K},\widehat{H}$ be arbitrary groups and let $K \leq \widehat{K}$ and $H \leq \widehat{H}$ be subgroups. Suppose that $\Omega$ is an $H$-set, that $\widehat{\Omega}$ is an $\widehat{H}$-set, and that there is an equivariant bijection $\xi: \Omega \longrightarrow \widehat{\Omega}$ such that $\xi(h\cdot \omega) = h \cdot \xi(\omega)$ for all $\omega \in \Omega$ and all $h\in H$. Then $K \wr_{\Omega} H$ embeds into $\widehat{K} \wr_{\widehat{\Omega}} \widehat{H}$ under the injective group homomorphism given by
	\begin{align*}
		\vartheta \colon &K \wr_{\Omega} H \hooklongrightarrow \widehat{K} \wr_{\widehat{\Omega}} \widehat{H}\\
		&\   (f, h) \longmapsto \big( \iota \circ f \circ \xi^{-1}, h \big),
	\end{align*}
    where $\iota: K \hooklongrightarrow \widehat{K}$ is the inclusion. In particular, when $\Omega = \widehat{\Omega}$, this allows us to consider $K \wr_{\Omega} H$ a subgroup of $\widehat{K} \wr_{\Omega} \widehat{H}$.
\end{corollary}

\begin{proof}
    This follows from the proof of Proposition \ref{Omega-wreath-isomorphism} by observing that if $\varphi: H \longrightarrow \widehat{H}$ and $\psi: K \longrightarrow \widehat{K}$ are both injective homomorphisms, then the map $\vartheta$ from Proposition \ref{Omega-wreath-isomorphism} is also an injective homomorphism, and using the natural inclusions $K \hooklongrightarrow \widehat{K}$ and $H \hooklongrightarrow \widehat{H}$ as injective homomorphisms.
\end{proof}

We will now show briefly two applications of these results to deduce the two theorems discussed at the beginning of Subsection \ref{Sharper-embeddings-subsection} of the Introduction, namely Theorems \ref{Chai-Oort-Theorem} and \ref{Barquero-Sanchez-Masri-Thorne-Theorem}. These two results will be deduced as consequences of Corollary \ref{Quadratic-embedding}.

\begin{proof}[Proof of Theorem \ref{Barquero-Sanchez-Masri-Thorne-Theorem}]

Let $K$ be a number field of degree $n$ and let $L$ be a quadratic extension of $K$. Let $\alpha$ be a primitive element for $K/\Q$, that is, write $K = \Q(\alpha)$ and let $\alpha_1:=\alpha, \dots, \alpha_n$ be the distinct conjugates of $\alpha$. Now, for each automorphism $\rho \in \Gal(K^c/\Q)$ we define a permutation $\pi_{\rho} \in S_n$, given by the relation $\rho(\alpha_i) = \alpha_{\pi_{\rho}(i)}$. Thus the set $G := \{ \pi_{\rho} \in S_d \mid \rho \in \Gal(K^c/\Q) \}$ is a transitive subgroup of $S_n$ and the map $\varphi \colon \Gal(K^c/\Q) \longrightarrow G$ given by $\varphi(\rho) = \pi_{\rho}$ is an isomorphism (see e.g. \cite[Proposition 6.3.1]{Cox12}).

Then the theorem asserts that the Galois group $\Gal(L^c/\Q)$ embeds as a subgroup of the wreath product $C_2 \wr_{\widehat{\Omega}} G$, where $\widehat{\Omega} := \{ 1, 2, \dots, n \}$ and $G$ acts on $\widehat{\Omega}$ in the natural way by permutations. To deduce this result from Corollary \ref{Quadratic-embedding}, observe that the corollary gives an embedding
$$
\Gal(L^c/\Q) \hooklongrightarrow \Gal(L/K) \wr_{\Omega} \Gal(K^c/\Q),
$$
where $\Omega := \operatorname{Hom}_{\Q}(K, K^c)$ is the set of embeddings $K \hooklongrightarrow K^c$. We can label the embeddings in $\Omega$ as $\operatorname{Hom}_{\Q}(K, K^c) = \{ \tau_1, \tau_2, \dots, \tau_n \}$, where $\tau_i \colon K \hooklongrightarrow K^c$ is given by $\tau_i(\alpha) := \alpha_i$ for every $i = 1, \dots, n$. Next, we define the bijection $\xi: \Omega \longrightarrow \widehat{\Omega}$ by $\xi(\tau_i) = i$. It remains to show that this bijection is equivariant under $\varphi$. Indeed, observe that
\begin{align*}
\xi(\rho \circ \tau_i) = \xi(\tau_{\pi_{\rho}(i)}) = \pi_{\rho}(i) = \varphi(\rho)(i) = \varphi(\rho)(\xi(\tau_i)).
\end{align*}
Then, by Proposition \ref{Omega-wreath-isomorphism} we conclude that $\Gal(L^c/\Q) \simeq C_2 \wr_{\widehat{\Omega}} G$, which is what we wanted to prove.
\end{proof}

\begin{remark}
In a similar manner, one can deduce Theorem \ref{Chai-Oort-Theorem} from Corollary \ref{Quadratic-embedding} in combination with Proposition \ref{Omega-wreath-isomorphism} and Corollary \ref{Subgroup-wreath-isomorphism}.
\end{remark}

\section{The Kaloujnine-Krasner Universal Embedding Theorem}\label{Kaloujnine-Krasner-section}

In 1951 Léo Kaloujnine\footnote{He was Russian and this is the transliteration he used for his name at the time when he wrote the papers on wreath products with Marc Krasner. His Russian name is nowadays transliterated as Lev Kaluznin.} and Marc Krasner published two very influential papers on wreath products for permutation groups and the theory of group extensions, titled \textit{``Produit complet des groupes de permutations et probl\`eme d'extension de groupes II and III"} (see \cite{KK51a} and \cite{KK51b}). In particular, in \cite{KK51a} they prove a fundamental result now known as the Kaloujnine-Krasner Universal Embedding Theorem. A generalization of the Universal Embedding Theorem of Kaloujnine and Krasner was given by D. F. Holt in \cite[Theorem 2]{Hol78}.

This theorem puts wreath products at the heart of the study of what is now known as the extension problem for groups. More specifically, in the theory of group extensions, the main objective is to study how a group $G$ can be constructed from knowledge of a normal subgroup $N \vartriangleleft G$ and knowledge of the corresponding quotient $G/N$. The interested reader can consult for example \cite[Chapter 11]{Rob96} or \cite[Chapter 7]{Rot95} to learn more about the theory of group extensions. 

Here we shall restrict ourselves to giving the definition of a group extension and a clear exposition of the theorem of Kaloujnine and Krasner. Let $N$ and $Q$ be arbitrary groups. Then, an \textit{extension} of $N$ by $Q$ is a group $G$ such that there is a normal subgroup $M \vartriangleleft G$ satisfying
\begin{itemize}
	\item[(i)] $M \simeq N$, and
	\item[(ii)] $G/M \simeq Q$.
\end{itemize}
Equivalently, we say that $G$ is an extension of $N$ by $Q$ if there is a short exact sequence of groups
$$
1 \longrightarrow N \overset{\iota}{\longrightarrow} G \overset{\varepsilon}{\longrightarrow} Q \longrightarrow 1.
$$
In this setting, we would have $M = \ker{\varepsilon} = \op{Im}(\iota)$.

Then, the Kaloujnine-Krasner theorem is the following. The proof that we give is adapted from the one sketched in \cite{Rot95}. Additionally, we exhibit the explicit embedding, which is then used at the end of the section to prove Theorem \ref{Main-Theorem}.

\begin{theorem}[Kaloujnine-Krasner]\label{Kaloujnine-Krasner-theorem}
Let $N$ and $Q$ be arbitrary groups. Then the regular wreath product $N \wr_r Q$ contains an isomorphic copy of every extension of $N$ by $Q$. More precisely, suppose that $G$ is a group extension of $N$ by $Q$ given by the exact sequence
$$
1 \longrightarrow N \overset{\iota}{\longrightarrow} G \overset{\varepsilon}{\longrightarrow} Q \longrightarrow 1.
$$
Then $G$ is isomorphic to a subgroup of the regular wreath product $N \wr_r Q$. In particular, if we identify $N$ with $\op{Im}(\iota) = \ker(\varepsilon) \vartriangleleft G$ and thus think of $N$ as a normal subgroup of $G$, then an explicit embedding is given by the map
\begin{align}\label{KK-varphi-definition}
\varphi \colon &G \longrightarrow N \wr_r Q\\
&g \longmapsto (\sigma_g, \varepsilon(g)) \notag,
\end{align}
where $\sigma_g \in N^{Q}$ is the map $\sigma_g \colon Q \longrightarrow N$ defined by
\begin{align}\label{KK-sigma-definition}
\sigma_g(q) := s(q)^{-1} g s \big( \varepsilon(g)^{-1}q \big)  
\end{align}
for every $q \in Q$.
\end{theorem}

\begin{proof}
	Suppose that $G$ is an extension of $N$ by $Q$ given by the exact sequence
$$
1 \longrightarrow N \overset{\iota}{\longrightarrow} G \overset{\varepsilon}{\longrightarrow} Q \longrightarrow 1.
$$
To make the notation simpler in the proof, we will identify $N$ with $\op{Im}(\iota) = \ker(\varepsilon) \vartriangleleft G$ and thus we will assume without loss of generality that $N$ is a normal subgroup of $G$.
Now, since $\varepsilon \colon G \to Q$ is surjective, let $s \colon Q \to G$ be a right inverse of $\varepsilon$. Thus $\varepsilon \circ s = \op{id}_Q$.

With this notation, define a map 
\begin{align}
\varphi \colon &G \longrightarrow N \wr_r Q\\
&g \longmapsto (\sigma_g, \varepsilon(g)) \notag,
\end{align}
where $\sigma_g \in N^{Q}$ is the map $\sigma_g \colon Q \longrightarrow N$ defined by
\begin{align}
\sigma_g(q) := s(q)^{-1} g s \big( \varepsilon(g)^{-1}q \big)  
\end{align}
for every $q \in Q$.

We will now prove that the map $\varphi$ is an injective homomorphism and that will finish the proof of the theorem.

First we observe that the map $\sigma_g$ is well defined, that is, that for every $q \in Q$ we have $\sigma_g(q) \in N$. To see this, note that $\sigma_g(q) \in \ker(\varepsilon) = N$. Indeed, using the fact that $\varepsilon$ is a homomorphism and that $\varepsilon \circ s = \op{id}_Q$, we have
\begin{align*}
	\varepsilon(\sigma_g(q)) &= \varepsilon(s(q))^{-1} \varepsilon(g) \varepsilon \big( s(\varepsilon(g)^{-1}q) \big) \\
	&= q^{-1} \varepsilon(g) \varepsilon(g)^{-1} q\\
	&= e_{Q}.
\end{align*}

\textit{Injectivity of $\varphi$:} Let $g_1, g_2 \in G$ and suppose that $\varphi(g_1) = \varphi(g_2)$. Then $(\sigma_{g_1}, \varepsilon(g_1)) = (\sigma_{g_2}, \varepsilon(g_2))$ and hence we have $\sigma_{g_1} = \sigma_{g_2}$ and $\varepsilon(g_1) = \varepsilon(g_2)$. Then, evaluating the functions $\sigma_{g_1}$ and $\sigma_{g_2}$ at $e_{Q}$ (to make things simpler) and using the fact that $\varepsilon(g_1) = \varepsilon(g_2)$, we see that
\begin{align*}
	\sigma_{g_1}(e_Q) = \sigma_{g_2}(e_Q) &\implies s(e_Q)^{-1} g_1 s \big( \varepsilon(g_1)^{-1} e_Q \big) = s(e_Q)^{-1} g_2 s \big( \varepsilon(g_2)^{-1} e_Q \big) \\
	&\implies g_1 s \big( \varepsilon(g_1)^{-1} \big) = g_2 s \big( \varepsilon(g_2)^{-1} \big) \\
	&\implies g_1 = g_2.
\end{align*}
In the previous calculation we have successively cancelled equal elements appearing on both sides of the equalities. Therefore the map $\varphi$ is injective.

\textit{The map $\varphi$ is a homomorphism:} Let $g_1, g_2 \in G$. Then we want to prove that 
$$
\varphi(g_1 g_2) = \varphi(g_1) \varphi(g_2).
$$
Note first that
\begin{align}\label{phi-LHS}
	\varphi(g_1 g_2) = (\sigma_{g_1 g_2}, \varepsilon(g_1 g_2)).	
\end{align}

On the other hand, using the definition (\ref{wreath-multiplication}) of the multiplication on the regular wreath product $N \wr_r Q$, we have
\begin{align}\label{phi-RHS}
	\varphi(g_1) \varphi(g_2) &= (\sigma_{g_1}, \varepsilon(g_1)) (\sigma_{g_2}, \varepsilon(g_2)) \notag \\
	&= (\sigma_{g_1} \theta_{\varepsilon(g_1)}(\sigma_{g_2}), \varepsilon(g_1) \varepsilon(g_2)) \notag \\ 
	&= (\sigma_{g_1} \theta_{\varepsilon(g_1)}(\sigma_{g_2}), \varepsilon(g_1 g_2)).
\end{align}
Thus, comparing the right hand side of (\ref{phi-LHS}) with the right hand side of (\ref{phi-RHS}), it remains to show that 
\begin{align}\label{sigma-equality}
	\sigma_{g_1} \theta_{\varepsilon(g_1)}(\sigma_{g_2}) = \sigma_{g_1 g_2}.
\end{align}

Now, in order to prove this, recall that $N \wr_r Q = N^Q \rtimes_{\theta} Q$, where
\begin{align*}
	\theta \colon &Q \longrightarrow \op{Aut}(N^Q)\\
	& q \longmapsto \theta_q
\end{align*}
and $\theta_q \colon N^Q \to N^Q$ is an automorphism, defined by $\theta_q(\sigma)(x) := \sigma(q^{-1} x)$ for every $\sigma \in N^{Q}$ and every $x \in Q$.
Thus, to prove (\ref{sigma-equality}), we let $x \in Q$ be arbitrary and evaluate to get
\begin{align*}
	\big( \sigma_{g_1} \theta_{\varepsilon(g_1)}(\sigma_{g_2}) \big)(x) &= \sigma_{g_1}(x) \big( \theta_{\varepsilon(g_1)}(\sigma_{g_2}) \big) (x) \\
	&= \sigma_{g_1}(x) \sigma_{g_2}(\varepsilon(g_1)^{-1} x) \\
	&= \Big( s(x)^{-1} g_1 s(\varepsilon(g_1)^{-1} x) \Big)  \Big( s \big( \varepsilon(g_1)^{-1} x \big)^{-1} g_2 s \big( \varepsilon(g_2)^{-1} \varepsilon(g_1)^{-1} x \big) \Big) \\
	&= s(x)^{-1} g_1 g_2 s \big( \varepsilon(g_1 g_2)^{-1} x \big) \\
	&= \sigma_{g_1 g_2}(x).
\end{align*}
Therefore, this proves (\ref{sigma-equality}) and this completes the proof that $\varphi(g_1 g_2) = \varphi(g_1) \varphi(g_2)$. Hence $\varphi \colon G \longrightarrow N \wr_r Q$ is an injective group homomorphism and this proves the theorem.
\end{proof}

\subsection{Proof of Theorem \ref{Main-Theorem}}

Let $F \subseteq K \subseteq L$ be a tower of fields with $L/F$ finite and separable. Then, from basic Galois theory we have the sequence of maps 
\begin{align}\label{theorem-1-sequence}
	1 \longrightarrow \Gal(L^c/K^c) \overset{\iota}{\longrightarrow} \Gal(L^c/F) \overset{\varepsilon}{\longrightarrow} \Gal(K^c/F) \longrightarrow 1,   
\end{align}
where $\iota$ is the inclusion and $\varepsilon$ is the restriction map, i.e., if $\rho \in \Gal(L^c/F)$, then $\varepsilon(\rho) = \rho|_{K^c}$. Moreover, by the Isomorphism Extension Theorem (see e.g. \cite[Theorem 3.20]{Mor96}), the map $\varepsilon$ is surjective and also we have 
\begin{align*}
\ker{(\varepsilon)} &= \{ \tau \in \Gal(L^c/F) \mid \tau|_{K^c} = \operatorname{id}_{K^c} \}\\
&= \Gal(L^c/K^c).
\end{align*}
This implies that the sequence \eqref{theorem-1-sequence} is a short exact sequence. Hence, by Theorem \eqref{Kaloujnine-Krasner-theorem} and equations \eqref{KK-varphi-definition} and \eqref{KK-sigma-definition}, we obtain the explicit embedding in Theorem \ref{Main-Theorem} and this concludes its proof.


\section{Proof of Theorem \ref{Main-Theorem-2}}\label{Main-Theorem-2-section}

We will now prove Theorem \ref{Main-Theorem-2}, which we restate here for the convenience of the reader.

\begin{theorem}
    Consider a tower of fields $F \subseteq K \subseteq L$ with $L/F$ finite and separable. Let $L^c$ be the Galois closure of $L/F$ and $K^c$ denote the Galois closure of $K/F$. Let $\Omega:= \op{Hom}_F(K, K^c)$, and let $s: \Omega \longrightarrow \Gal(L^c/F)$ be a map that extends every $F$-embedding $\omega \in \Omega$ to an automorphism $s(\omega) \in \Gal(L^c/F)$. Define the map
    \begin{align*}
		\varphi \colon &\Gal(L^c/F) \hooklongrightarrow \Gal(L^c/K) \wr_{\Omega} \Gal(K^c/F) \\
		&\quad \quad \qquad  \rho \longmapsto (\sigma_{\rho}, \varepsilon(\rho)), \notag
	\end{align*}
    where $\sigma_{\rho}: \Omega = \op{Hom}_F(K, K^c) \longrightarrow \Gal(L^c/K)$ is given by 
    \begin{align*}
    \sigma_\rho(\omega) = s(\omega)^{-1} \circ \rho \circ s(\varepsilon(\rho)^{-1}\circ \omega)
    \end{align*}
    and $\varepsilon$ is the the restriction from  $\Gal(L^c/F)$ to $\Gal(K^c/F)$. Then, the map $\varphi$ is an injective group homomorphism, embedding the group $\Gal(L^c/F)$ into the wreath product $\Gal(L^c/K) \wr_{\Omega} \Gal(K^c/F)$.
\end{theorem}

\begin{proof}
We first have to check that the map $\varphi$ is well-defined, that is, that for every $\rho \in \Gal(L^c/F)$ and every $\omega \in \Omega = \op{Hom}_F(K, K^c)$ we have
$$
\sigma_{\rho}(\omega) \in \Gal(L^c/K).
$$
Observe first that $\varepsilon(\rho)^{-1} \circ \omega \in \Omega$. Hence $s(\varepsilon(\rho)^{-1} \circ \omega) \in \Gal(L^c/F)$. Then, since also $\rho$ and $s(\omega)$ are in $\Gal(L^c/F)$, the composition of automorphisms defining $\sigma_{\rho}(\omega)$ satisfies 
$$
\sigma_\rho(\omega) = s(\omega)^{-1} \circ \rho \circ s(\varepsilon(\rho)^{-1}\circ \omega) \in \Gal(L^c/F).
$$
Now we want to prove that $\sigma_{\rho}(\omega)$ leaves $K$ fixed, so that we can conclude that $\sigma_{\rho}(\omega) \in \Gal(L^c/K)$.

Thus, let $k \in K$. Then 
\begin{align*}
\sigma_{\rho}(\omega)(k) &= \big( s(\omega)^{-1} \circ \rho \circ s(\varepsilon(\rho)^{-1}\circ \omega) \big)(k)\\
&= s(\omega)^{-1} \Big( \rho \big( s(\varepsilon(\rho)^{-1} \circ \omega)(k) \big) \Big) \\
&= s(\omega)^{-1} \Big( \rho \big( \varepsilon(\rho)^{-1}(\omega(k)) \big)  \Big)  &&\text{(since $k \in K$ we can omit $s$)} \\
&= s(\omega)^{-1} \Big( \varepsilon(\rho) \big( \varepsilon(\rho)^{-1}(\omega(k)) \big)  \Big)  &&\text{(since $\varepsilon(\rho)^{-1}(\omega(k)) \in K^c$ we can change $\rho$ to $\varepsilon(\rho)$)}\\
&= s(\omega)^{-1}\big( \omega(k) \big)\\
&= k. &&\text{(we used the fact that $\omega(k) = s(\omega)(k)$)}
\end{align*}
Therefore, this shows that $\sigma_{\rho}(\omega) \in \Gal(L^c/K)$ and this completes the proof that $\varphi$ is well-defined. Now we proceed to prove that $\varphi$ is an injective group homomorphism.

\noindent $\blacklozenge$ \textit{$\varphi$ is a group homomorphism:} Let $\rho_1, \rho_2 \in \Gal(L^c/F)$. We will prove that
$$
\varphi(\rho_1 \circ \rho_2) = \varphi(\rho_1) \varphi(\rho_2).
$$
Note first that 
\begin{align*}
    \varphi(\rho_1\circ \rho_2) &= (\sigma_{\rho_1\circ \rho_2}, \varepsilon(\rho_1 \circ \rho_2))\\
    &= (\sigma_{\rho_1\circ \rho_2},(\rho_1 \circ \rho_2) |_{K^c}) \\
    &= (\sigma_{\rho_1\circ \rho_2}, \rho_1 |_{K^c} \circ \rho_2 |_{K^c}).
\end{align*}
On the other hand, by the definition of multiplication in the wreath product $\Gal(L/K) \wr_{\Omega} \Gal(K^c/F)$, we have
\begin{align*}
    \varphi(\rho_1)\varphi(\rho_2) &= (\sigma_{\rho_1}, \varepsilon(\rho_1))(\sigma_{\rho_2}, \varepsilon(\rho_2)) \\
    &= (\sigma_{\rho_1} \theta_{\varepsilon(\rho_1)}(\sigma_{\rho_2}), \varepsilon(\rho_1) \circ \varepsilon(\rho_2)) \\
    &= (\sigma_{\rho_1} \theta_{\varepsilon(\rho_1)}(\sigma_{\rho_2}), \rho_1|_{K^c} \circ \rho_2|_{K^c}).
\end{align*}
Thus, we need to show that
\begin{align}\label{Main-homomorphism-2-proof}
     \sigma_{\rho_1} \theta_{\varepsilon(\rho_1)}(\sigma_{\rho_2}) = \sigma_{\rho_1\circ\rho_2}.
\end{align}

Hence, let $\omega \in \Omega$. Then, evaluating the LHS of (\ref{Main-homomorphism-2-proof}) at $\omega$ we have
\begin{align*}
\big( \sigma_{\rho_1} \theta_{\varepsilon(\rho_1)}(\sigma_{\rho_2}) \big)(\omega) &= \sigma_{\rho_1}(\omega) \circ \sigma_{\rho_2}(\varepsilon(\rho_1)^{-1} \circ \omega)\\
&= s(\omega)^{-1} \circ \rho_1 \circ s(\varepsilon(\rho_1)^{-1}\circ \omega) \circ s(\varepsilon(\rho_1)^{-1}\circ \omega)^{-1} \circ \rho_2 \circ s(\varepsilon(\rho_2)^{-1} \circ \varepsilon(\rho_1)^{-1}\circ \omega)\\
&= s(\omega)^{-1} \circ \big( \rho_1 \circ \rho_2 \big) \circ s( \varepsilon(\rho_1 \circ \rho_2)^{-1}\circ \omega)\\
&= \sigma_{\rho_1 \circ \rho_2}(\omega).
\end{align*}
This proves that $\varphi$ is a group homomorphism.

\noindent $\blacklozenge$ \textit{$\varphi$ is injective:} Let $\rho_1, \rho_2 \in \Gal(L^c/F)$ and suppose that $\varphi(\rho_1) = \varphi(\rho_2)$. Then 
$$
(\sigma_{\rho_1}, \varepsilon(\rho_1)) = (\sigma_{\rho_2}, \varepsilon(\rho_2)),
$$
that is, we have $\sigma_{\rho_1} = \sigma_{\rho_2}$ and $\varepsilon(\rho_1) = \varepsilon(\rho_2)$. Now, since $\sigma_{\rho_1} = \sigma_{\rho_2}$, then for every $\omega \in \Omega$ we have
\begin{align*}
s(\omega)^{-1} \circ \rho_1 \circ s(\varepsilon(\rho_1)^{-1} \circ \omega) = s(\omega)^{-1} \circ \rho_2 \circ s(\varepsilon(\rho_2)^{-1} \circ \omega)
\end{align*}
and cancelling the terms $s(\omega)^{-1}$ on the left and the terms $s(\varepsilon(\rho_2)^{-1} \circ \omega)$ on the right of the previous equality, we conclude that $\rho_1 = \rho_2$, thus proving that $\varphi$ is injective.

This completes the proof of the theorem.
\end{proof}

\section{Galois groups and wreath products for Kummer extensions}\label{Kummer-section}

In this section we give a proof of Theorem \ref{Kummer-embedding}. Before we embark on the proof though, we will need the following technical lemma.

\begin{lemma}\label{Galois-closure-2}
Let $F \subseteq K \subseteq L$ be a tower of fields with $L/F$ finite and separable. Suppose that $K$ contains a primitive $n$-th root of unity $\omega$. Denote by $K^c$ and $L^c$ the Galois closures of $K/F$ and $L/F$, respectively. Then, if $L = K(\sqrt[\leftroot{-3}\uproot{3} n_1]{\alpha_1},\cdots, \sqrt[\leftroot{-3}\uproot{3} n_r]{\alpha_r})$ for some $\alpha_j \in K$ and some positive integers $n_j$ dividing $n$, we have that the Galois closure $L^c$ is given by
$$
L^c = K^c(X),
$$
where $X := \{ \sqrt[\leftroot{-3}\uproot{3} n_j]{\tau(\alpha_j)} \suchthat \tau \in \op{Hom}_F(K, K^c) \text{ and } j\in\{1,\cdots,r\}  \}$.
\end{lemma}

\begin{proof}
    Since $L/F$ is finite and separable, also $K/F$ is separable. By the Primitive Element Theorem, there is an element $\beta \in K$ such that $K = F(\beta)$. If $\beta = \beta_1,\beta_2,\cdots,\beta_d$ are the Galois conjugates of $\beta$, then $K^c = F(\beta_1,\cdots,\beta_d)$ and this implies that $L= K(\sqrt[\leftroot{-3}\uproot{3} n_1]{\alpha_1},\cdots,\sqrt[\leftroot{-3}\uproot{3} n_r]{\alpha_r} ) = F(\beta, \sqrt[\leftroot{-3}\uproot{3} n_1]{\alpha_1},\cdots,\sqrt[\leftroot{-3}\uproot{3} n_r]{\alpha_r})$.
    
    Now, if $\alpha_j = \alpha_{1,j}, \alpha_{2,j},\cdots, \alpha_{m_j,j}$ are the Galois conjugates of $\alpha_j$, then 
    $$
    \widetilde{X_j} := \{\omega^\ell \sqrt[\leftroot{-3}\uproot{6} n_j]{\alpha_{i,j}} \suchthat  1\leq i \leq m_j, 0 \leq \ell <n \}
    $$ 
    is the set of Galois conjugates of $\sqrt[\leftroot{-3}\uproot{5} n_j]{\alpha_j}$, and therefore we have 
    $$
    L^c = F(\beta_1,\cdots,\beta_d, \widetilde{X_1},\cdots, \widetilde{X_r}) = K^c(\widetilde{X_1},\cdots,\widetilde{X_r}).
    $$
    Then, since $K$ contains a primitive $n$-th root of unity $\omega$, all the powers $\omega^{\ell}$ are already in $K$, so we actually have
    $$
    K^c(\widetilde{X_1},\cdots,\widetilde{X_r}) = K^c(\widehat{X_1},\cdots,\widehat{X_r}),$$
    where
    $$
    \widehat{X_j} := \{ \sqrt[\leftroot{-3}\uproot{6} n_j]{\alpha_{i,j}} \suchthat  1\leq i \leq m_j \}.
    $$
    Now, if $\tau \in \op{Hom}_F(K,K^c)$, then for each $j\in\{1,\cdots,r\}$ we have $\tau(\alpha_j) = \tau(\alpha_{1,j}) = \alpha_{i,j}$ for some $i \in \{1,\cdots,m_j\}$. Therefore, this gives us that  $X = \widehat{X_1} \cup \cdots \cup \widehat{X_r}$, and this means that $L^c=K^c(X)$. 
\end{proof}

With this preliminary result, we are now ready to prove Theorem \ref{Kummer-embedding}, which we restate here for the convenience of the reader.

\begin{theorem}
Let $F \subseteq K \subseteq L$ be a tower of fields with $L/F$ finite and separable. Moreover, suppose that $K$ contains a primitive $n$-th root of unity $\omega$ and that $L/K$ is an $n$-Kummer extension. As in Theorem \ref{Kummer-extensions}, write $L = K(\sqrt[\leftroot{-3}\uproot{3} n_1]{\alpha_1}, \dots, \sqrt[\leftroot{-3}\uproot{3} n_r]{\alpha_r})$ for some $\alpha_i \in K$ and some positive integers $n_i$ dividing $n$, and put $\Gal(L/K) = \langle \eta_1 \rangle \cdots \langle \eta_r \rangle$, where the order of $\eta_i$ equals $n_i$ for each $i\in \{1,2,\dots, r\}$. Denote the Galois closures of $K/F$ and $L/F$ by $K^c$ and $L^c$, respectively and let $\Omega = \op{Hom}_{F}(K, K^c)$ be the set of $F$-homomorphisms $K \hookrightarrow K^c$. The group $\Gal(K^c/F)$ acts on $\Omega$ on the left by composition of maps. Then the Galois group $\Gal(L^c/F)$ embeds into the wreath product $\Gal(L/K) \wr_{\Omega} \Gal(K^c/F)$. More precisely, consider the map 
	\begin{align}
		\varphi \colon &\Gal(L^c/F) \longrightarrow \Gal(L/K) \wr_{\Omega} \Gal(K^c/F) \\
		&\quad \quad \qquad  \rho \longmapsto (\sigma_{\rho}, \varepsilon(\rho)) \notag
	\end{align}
	where $\sigma_{\rho} \colon \op{Hom}_F(K, K^c) \longrightarrow \Gal(L/K)$ is defined for every embedding $\tau \in \op{Hom}_F(K, K^c)$ by
	$$
	   \sigma_{\rho}(\tau) = \eta_1^{\chi(\rho, \tau,\alpha_1)}\circ \cdots \circ \eta_r^{\chi(\rho, \tau,\alpha_r)},
	$$
with $\chi(\rho, \tau, \alpha_j) \in \Z/n_j\Z$ chosen to satisfy 
    $$
      \tau(\omega_j)^{\chi(\rho, \tau, \alpha_j)} = \frac{\rho\left(\sqrt[\leftroot{-3}\uproot{3} n_j]{\rho^{-1}(\tau(\alpha_j))}\right)}{\sqrt[\leftroot{-3}\uproot{3} n_j]{\tau(\alpha_j)}},
    $$
where $\omega_j = \omega^{n/n_j}$ for $j \in \{1,2,\cdots,r\}$. Then the map $\varphi$ is an injective group homomorphism.
\end{theorem}

\begin{proof}


Let $X := \{ \sqrt[\leftroot{-3}\uproot{3} n_j]{\tau(\alpha_j)} \suchthat \tau \in \op{Hom}_F(K, K^c) \text{ and } j\in\{1,2,\cdots,r\}\} \subseteq K^c$. Then, by Lemma \ref{Galois-closure-2} we have $L^c = K^c(X)$. Now, for each $\rho \in \Gal(L^c/F)$  and $j\in \{1,\cdots,r\}$ we have
\begin{align*}
    \rho\left(\sqrt[\leftroot{-3}\uproot{3} n_j]{\tau(\alpha_j)}\right)^{n_j} = \rho(\tau(\alpha_j)).
\end{align*}
This means that $\rho(\sqrt[\leftroot{-3}\uproot{3} n_j]{\tau(\alpha_j)})$ is an $n_j$-th root of $\rho(\tau(\alpha_j))$. We note that $\omega_j = \omega^{n/n_j}$ is actually a primitive $n_j$-th root of unity in $K$. It follows from this that
\begin{align*}
\rho\left(\sqrt[\leftroot{-3}\uproot{3} n_j]{\tau(\alpha_j)}\right) = \omega_j^{k}\sqrt[\leftroot{-3}\uproot{3} n_j]{\rho(\tau(\alpha_j))}
\end{align*}
for some $k \in \{ 0, 1, \dots, n_j - 1 \}$. Furthermore, we observe that for each $\rho \in \Gal(L^c/F)$ and each $\tau \in \Omega = \operatorname{Hom}_F(K, K^c)$, we have that $(\rho \circ \tau)(\omega_j)$ is also a primitive $n_j$-th root of unity in $K$.\footnote{This is a consequence of the fact that if $f: K_1 \to K_2$ is a homomorphism of fields, then $f$ preserves the multiplicative order of elements in $K_1^{\times}$ 


} Hence, there is a $k' \in \{ 0, 1, \dots, n_j-1 \}$ such that $\omega_j^k = (\rho \circ \tau)(\omega_j)^{k'}$. Thus we have
\begin{align}\label{k_prime_eqn}
\rho\left(\sqrt[\leftroot{-3}\uproot{3} n_j]{\tau(\alpha_j)}\right) = (\rho \circ \tau)(\omega_j)^{k'}\sqrt[\leftroot{-3}\uproot{3} n_j]{\rho(\tau(\alpha_j))}.
\end{align}
Therefore, for each $\rho \in \Gal(L^c/F)$ and $\tau \in \op{Hom}_F(K, K^c)$, using equation \eqref{k_prime_eqn} with $\tau$ replaced by $\rho^{-1} \circ \tau$, we define $\chi(\rho, \tau, \alpha_j) \in \Z/n_j\Z$ so that the equation
\begin{align}\label{chi-definition-equation}
\rho \left( \sqrt[\leftroot{-3}\uproot{3} n_j]{\rho^{-1}(\tau(\alpha_j))} \right) = \tau(\omega_j)^{\chi(\rho, \tau, \alpha_j)}\sqrt[\leftroot{-3}\uproot{3} n_j]{\tau(\alpha_j)}   
\end{align}
is satisfied.

Now, consider the map
\begin{align*}
	\varphi \colon &\Gal(L^c/F) \longrightarrow \Gal(L/K) \wr_{\Omega} \Gal(K^c/F) \\
	&\quad \quad \qquad  \rho \longmapsto (\sigma_{\rho}, \varepsilon(\rho)) \notag
\end{align*}
where $\varepsilon(\rho) = \rho|_{K^c}$ is the restriction of $\rho$ to $K^c$ and $\sigma_{\rho} \colon \op{Hom}_F(K, K^c) \longrightarrow \Gal(L/K)$ is defined for every embedding $\tau \in \op{Hom}_F(K, K^c)$ by
$$
    \sigma_{\rho}(\tau) = \eta_1^{\chi(\rho, \tau, \alpha_1)}\circ \cdots \circ \eta_r^{\chi(\rho, \tau, \alpha_r)}.
$$
Observe that in the wreath product $\Gal(L/K) \wr_{\Omega} \Gal(K^c/F)$, the group $\Gal(K^c/F)$ acts on the set of embeddings $\Omega = \op{Hom}_F(K, K^c)$ by composition, i.e., if $\xi \in \Gal(K^c/F)$ and $\tau \in \Omega$, then the action is given by $\xi \circ \tau$.
We are now ready to prove that $\varphi$ is an injective group homomorphism.

\noindent $\blacklozenge$ \textit{$\varphi$ is a group homomorphism:} Let $\rho_1, \rho_2 \in \Gal(L^c/F)$. We will prove that
$$
\varphi(\rho_1 \circ \rho_2) = \varphi(\rho_1) \varphi(\rho_2).
$$
Note first that 
\begin{align*}
    \varphi(\rho_1\circ \rho_2) &= (\sigma_{\rho_1\circ \rho_2}, \varepsilon(\rho_1 \circ \rho_2))\\
    &= (\sigma_{\rho_1\circ \rho_2},(\rho_1 \circ \rho_2) |_{K^c}) \\
    &= (\sigma_{\rho_1\circ \rho_2}, \rho_1 |_{K^c} \circ \rho_2 |_{K^c}).
\end{align*}
On the other hand, by the definition of multiplication in the wreath product $\Gal(L/K) \wr_{\Omega} \Gal(K^c/F)$, we have
\begin{align*}
    \varphi(\rho_1)\varphi(\rho_2) &= (\sigma_{\rho_1}, \varepsilon(\rho_1))(\sigma_{\rho_2}, \varepsilon(\rho_2)) \\
    &= (\sigma_{\rho_1} \theta_{\varepsilon(\rho_1)}(\sigma_{\rho_2}), \varepsilon(\rho_1) \circ \varepsilon(\rho_2)) \\
    &= (\sigma_{\rho_1} \theta_{\varepsilon(\rho_1)}(\sigma_{\rho_2}), \rho_1|_{K^c} \circ \rho_2|_{K^c}).
\end{align*}
Thus, we need to show that
\begin{align}\label{cyclic-homomorph_main_kummer}
     \sigma_{\rho_1\circ\rho_2} = \sigma_{\rho_1} \theta_{\varepsilon(\rho_1)}(\sigma_{\rho_2}).
\end{align}
Note that both sides of \eqref{cyclic-homomorph_main_kummer} are functions in the group $\Gal(L/K)^{\Omega}$, that is, they are functions $\Omega \longrightarrow \Gal(L/K)$. Hence, in order to show that they are equal, we will show that both sides of (\ref{cyclic-homomorph_main_kummer}) have the same value at each $\tau \in \Omega = \operatorname{Hom}_F(K, K^c)$. 

For the right hand side of (\ref{cyclic-homomorph_main_kummer}), note that
\begin{align}\label{cyclic-homomorph-aux-2-kummer}
    \left(\sigma_{\rho_1} \theta_{\varepsilon(\rho_1)}(\sigma_{\rho_2})\right)(\tau) &= \sigma_{\rho_1}(\tau) \circ \theta_{\varepsilon(\rho_1)}(\sigma_{\rho_2}(\tau)) \notag \\
    &= \sigma_{\rho_1}(\tau) \circ \sigma_{\rho_2}\left(\varepsilon(\rho_1)^{-1}\circ \tau\right) \notag \\
    &= \eta_1^{\chi(\rho_1,\tau, \alpha_1)} \circ \cdots \circ \eta_r^{\chi(\rho_1,\tau, \alpha_r)} \circ \eta_1^{\chi(\rho_2, \varepsilon(\rho_1)^{-1}\circ \tau, \alpha_1)}\circ \eta_r^{\chi(\rho_2, \varepsilon(\rho_1)^{-1}\circ \tau, \alpha_r)} \notag\\
    &= \eta_1^{\chi(\rho_1,\tau, \alpha_1)} \circ \cdots \circ \eta_r^{\chi(\rho_1,\tau, \alpha_r)} \circ \eta_1^{\chi(\rho_2, \rho_1^{-1}\circ \tau, \alpha_1)}\circ \eta_r^{\chi(\rho_2, \rho_1^{-1}\circ \tau, \alpha_r)} \notag \\
     &= \eta_1^{\chi(\rho_1,\tau, \alpha_1) + \chi(\rho_2, \rho_1^{-1}\circ \tau, \alpha_1)} \circ \cdots \circ \eta_r^{\chi(\rho_1,\tau, \alpha_r) + \chi(\rho_2, \rho_1^{-1}\circ \tau, \alpha_r)},
\end{align}
where we used that $\Gal(L/K)$ is abelian. On the other hand, the LHS of \eqref{cyclic-homomorph_main_kummer} evaluated at $\tau$ is given by
$$
\sigma_{\rho_1 \circ \rho_2}(\tau) = \eta_1^{\chi(\rho_1 \circ \rho_2, \tau, \alpha_1)}\circ \cdots \circ \eta_r^{\chi(\rho_1 \circ \rho_2, \tau, \alpha_r)}.
$$
Therefore, we will now prove that 
\begin{align}\label{cocycle-relation}
\chi(\rho_1 \circ \rho_2, \tau, \alpha_j) = \chi(\rho_2, \rho_1^{-1}\circ \tau, \alpha_j) + \chi(\rho_1,\tau, \alpha_j) 
\end{align}
for all $j \in \{1,2,\cdots,r\}$. Indeed, note that for each $j \in \{1,2,\cdots,r\}$ we have
\begin{align*}
    (\rho_1 \circ \rho_2) \left(\sqrt[\leftroot{-3}\uproot{3} n_j]{(\rho_1\circ \rho_2)^{-1}(\tau(\alpha_j))}\right) 
    &= \rho_1 \left[ \rho_2 \left( \sqrt[\leftroot{-3}\uproot{3} n_j]{\rho_2^{-1}\left((\rho_1^{-1}\circ \tau)(\alpha_j)\right)}\right)\right] \\
    &= \rho_1 \left[ (\rho_1^{-1}\circ \tau)(\omega_j)^{\chi(\rho_2, \rho_1^{-1}\circ \tau, \alpha_j)}  \left(\sqrt[\leftroot{-3}\uproot{3} n_j]{(\rho_1^{-1}\circ \tau)(\alpha_j)}\right)\right]\\
    & = \tau(\omega_j)^{\chi(\rho_2, \rho_1^{-1}\circ \tau, \alpha_j)}  \rho_1 \left(\sqrt[\leftroot{-3}\uproot{3} n_j]{\rho_1^{-1}( \tau(\alpha_j))}\right)\\
    &= \tau(\omega_j)^{\chi(\rho_2, \rho_1^{-1}\circ \tau, \alpha_j)}  \tau(\omega_j)^{\chi(\rho_1,\tau, \alpha_j)}\sqrt[\leftroot{-3}\uproot{3} n_j]{\tau(\alpha_j)} \\
    &= \tau(\omega_j)^{\chi(\rho_2, \rho_1^{-1}\circ \tau, \alpha_j) + \chi(\rho_1,\tau, \alpha_j) }\sqrt[\leftroot{-3}\uproot{3} n_j]{\tau(\alpha_j)}.
\end{align*}
Using the defining equation \eqref{chi-definition-equation} for $\chi(\rho_1 \circ \rho_2, \tau, \alpha_j)$,  the previous calculation implies that
$$
\chi(\rho_1 \circ \rho_2, \tau, \alpha_j) = \chi(\rho_2, \rho_1^{-1}\circ \tau, \alpha_j) + \chi(\rho_1,\tau, \alpha_j).
$$
Hence, from this we conclude that
\begin{align}\label{cyclic-homomorph-aux-1-kummer}
\sigma_{\rho_1 \circ \rho_2}(\tau) = \eta_1^{\chi(\rho_2, \rho_1^{-1}\circ \tau, \alpha_1) + \chi(\rho_1,\tau, \alpha_1)}\circ \cdots \circ \eta_r^{\chi(\rho_2, \rho_1^{-1}\circ \tau, \alpha_r) + \chi(\rho_1,\tau, \alpha_r)}.
\end{align}

Therefore, comparing equations (\ref{cyclic-homomorph-aux-1-kummer}) and (\ref{cyclic-homomorph-aux-2-kummer}), we see that (\ref{cyclic-homomorph_main_kummer}) holds, thus completing the proof that $\varphi(\rho_1 \circ \rho_2) = \varphi(\rho_1) \varphi(\rho_2)$.
 
\noindent $\blacklozenge$ \textit{$\varphi$ is injective:} Since we already know that $\varphi$ is a group homomorphism, to prove injectivity we will show that $\ker{\varphi} = \{ \op{id}_{L^c} \}$. Thus, let $\rho \in \ker{\varphi}$. Hence $\rho \in \Gal(L^c/F)$ and $\varphi(\rho) = (\sigma_{\rho}, \varepsilon(\rho))$ is the identity in $\Gal(L/K)\wr_{\Omega} \Gal(K^c/F)$. In other words, this means that $\varepsilon(\rho) = \rho|_{K^c} = \op{id}_{K^c}$ and that $\sigma_{\rho}$ is the identity element in $\Gal(L/K)^{\Omega}$, that is, that for every $\tau \in \Omega = \op{Hom}_F(K, K^c)$ we have
$$
\sigma_{\rho}(\tau) = \eta_1^{\chi(\rho, \tau, \alpha_1)} \circ \cdots \circ \eta_r^{\chi(\rho, \tau, \alpha_r)} = \op{id}_{L}.
$$
Now, recall that $L^c = K^c(X)$, where $$X = \left \{ \sqrt[\leftroot{-3}\uproot{3} n_j]{\tau(\alpha_j)} \suchthat \tau \in \op{Hom}_F(K, K^c) \text{ and } j\in\{1,\cdots,r\}  \right \}.$$ 
This implies that the automorphism $\rho \in \Gal(L^c/F)$ is completely determined by its action on the elements of $K^c$ and on the elements of the generating set $X$. We want to prove that $\rho = \operatorname{id}_{L^c}$. We already know that $\rho|_{K^c} = \rho = \operatorname{id}_{K^c}$. Hence, it suffices to show that $\rho$ acts as the identity map on the elements of the generating set $X$.

To do this, first remember that $\Gal(L/K) = \langle \eta_1 \rangle \cdots \langle \eta_r \rangle \cong \langle \eta_1 \rangle \times \cdots \times \langle \eta_r \rangle$. Since this is a direct product, this implies that $\eta_j^{\chi(\rho, \tau, \alpha_j)} = \op{id}_L$ for each $j \in \{1,\cdots,r\}$. Then, because $\chi(\rho,\tau,\alpha_j) \in \Z/n_j\Z$ and the order of $\eta_j$ in $\Gal(L/K)$ equals $n_j$, this gives us $\chi(\rho,\tau,\alpha_j) = 0$ for every $j$ and every $\tau$. Hence, from the definition of $\chi(\rho, \tau, \alpha_j)$, it follows that
$$
\rho(\sqrt[\leftroot{-3}\uproot{3} n_j]{\rho^{-1}(\tau(\alpha_j))}) = \tau(\omega_j)^{\chi(\rho, \tau, \alpha_j)}\sqrt[\leftroot{-3}\uproot{3} n_j]{\tau(\alpha_j)} = \sqrt[\leftroot{-3}\uproot{3} n_j]{\tau(\alpha_j)}.
$$
Moreover, since $\rho|_{K^c} = \op{id}_{K^c}$ and $\tau(\alpha_j) \in K^c$, it follows that $\rho^{-1}(\tau(\alpha_j)) = \tau(\alpha_j)$. Therefore, this implies that
$$
\rho\left(\sqrt[\leftroot{-3}\uproot{3} n_j]{\tau(\alpha_j)}\right) =  \sqrt[\leftroot{-3}\uproot{3} n_j]{\tau(\alpha_j)}
$$
for every $\tau \in \op{Hom}_F(K, K^c)$. Thus, this proves that $\rho$ acts as the identity on the elements of the generating set $X$.

Hence, we conclude that $\rho = \op{id}_{L^c}$. Consequently, $\ker{\varphi} = \{ \op{id}_{L^c} \}$ and this completes the proof that $\varphi$ is injective.
\end{proof}


\section{Examples of Theorems \ref{Main-Theorem}, \ref{Main-Theorem-2} and \ref{Kummer-embedding}}\label{Examples-section}

In this section we consider some concrete examples of the embeddings from Theorems \ref{Main-Theorem}, \ref{Main-Theorem-2} and \ref{Kummer-embedding}.

\begin{example}[Example of Theorem \ref{Main-Theorem}]\label{Biquadratic-example}
	Consider the number fields $K = \Q(\sqrt{5})$ and $L = K(\sqrt{7}) = \Q(\sqrt{5}, \sqrt{7})$. Then both extensions $K/\Q$ and $L/\Q$ are Galois and $L = \Q(\sqrt{5}, \sqrt{7})$ is a biquadratic number field. Thus we have $K^c = K$ and $L^c = L$ and therefore we have the short exact sequence
	$$
		1 \longrightarrow \Gal(L/K) \overset{\iota}{\longrightarrow} \Gal(L/\Q) \overset{\varepsilon}{\longrightarrow} \Gal(K/\Q) \longrightarrow 1,
	$$
	
	The three Galois groups we will consider are given as follows. First, for the extension $K/\Q$ we have $\Gal(K/\Q) = \{ \op{id}_K, \eta  \}$, where $\eta \colon K \to K$ is given by $\eta(a + b\sqrt{5}) = a - b\sqrt{5}$ for every $a, b \in \Q$. 
	
	Now, for the biquadratic extension $L/\Q$, the Galois group is given by $\Gal(L/\Q) = \{ \op{id}_L, \rho_1, \rho_2, \rho_3 \}$, where the automorphisms $\rho_i \colon L \to L$ are given by
	\begin{align*}
		\rho_1 &\colon \sqrt{5} \mapsto -\sqrt{5}, \quad \sqrt{7}
			\mapsto \sqrt{7},\\
		\rho_2 &\colon \sqrt{5} \mapsto \sqrt{5}, \quad -\sqrt{7} \mapsto \sqrt{7},\\
		\rho_3 &\colon \sqrt{5} \mapsto -\sqrt{5}, \quad \sqrt{7} \mapsto -\sqrt{7}.
	\end{align*}
	We note that $\Gal(L/\Q) \simeq \Z/2\Z \oplus \Z/2\Z$ and in particular we have $\rho_1^2 = \rho_2^2 = \rho_3^2 = \op{id}_L$ and $\rho_1 \circ \rho_2 = \rho_3$, $\rho_2 \circ \rho_3 = \rho_1$ and $\rho_1 \circ \rho_3 = \rho_2$. On the other hand, for the extension $L/K$ the Galois group is $\Gal(L/K) = \{ \op{id}_L, \rho_2 \}$.
		
	With this, we can describe the elements of the regular wreath product 
	$$
	\Gal(L/K) \wr_r \Gal(K/\Q) = \Gal(L/K)^{\Gal(K/\Q)} \rtimes_{\theta} \Gal(K/\Q),
	$$
	which is a group of order $|\Gal(L/K)|^{|\Gal(K/\Q)|} \cdot  |\Gal(K/\Q)| = 2^2 \cdot 2 = 8$.

	First, the group $\Gal(L/K)^{\Gal(K/\Q)}$ of functions $f \colon \Gal(K/\Q) \longrightarrow \Gal(L/K)$ has four elements, which are given by
	\begin{align*}
		f_1 &\colon \op{id}_K \mapsto \op{id}_L, \quad \eta \mapsto \op{id}_L,\\
		f_2 &\colon \op{id}_K \mapsto \op{id}_L, \quad \eta \mapsto \rho_2, \\
		f_3 &\colon \op{id}_K \mapsto \rho_2, \quad \eta \mapsto \op{id}_L,\\
		f_4 &\colon \op{id}_K \mapsto \rho_2, \quad \eta \mapsto \rho_2.
	\end{align*}
	We observe that $\Gal(L/K)^{\Gal(K/\Q)} = \{ f_1, f_2, f_3, f_4 \} \simeq \Z/2\Z \oplus \Z/2\Z$ with $f_2^2 = f_3^2 = f_4^2 = f_1$ and $f_2 f_3 = f_4$, $f_2 f_4 = f_3$ and $f_3 f_4 = f_2$.
	
	Therefore, as a set, we have
	$$
	\Gal(L/K) \wr_r \Gal(K/\Q) = \{ (f_1, \op{id}_K), (f_2, \op{id}_K), (f_3, \op{id}_K), (f_4, \op{id}_K), (f_1, \eta), (f_2, \eta), (f_3, \eta), (f_4, \eta) \}.
	$$
	
	Now, recall that the surjective map $\varepsilon \colon \Gal(L/\Q) \longrightarrow \Gal(K/\Q)$ is given by $\varepsilon(\rho) = \rho|_K$ for every $\rho \in \Gal(L/\Q)$. Therefore, we have $\varepsilon(\op{id}_L) = \varepsilon(\rho_2) = \op{id}_K$ and $\varepsilon(\rho_1) = \varepsilon(\rho_3) = \eta$. In order to define the embedding (\ref{Galois-embedding-2}) from Theorem \ref{Main-Theorem} we need a right inverse $s$ of $\varepsilon$. Thus, consider the map $s \colon \Gal(K/\Q) \longrightarrow \Gal(L/\Q)$ given by $s(\op{id}_K) = \op{id}_L$ and $s(\eta) = \rho_1$. Then we have $\varepsilon \circ s = \op{id}_{\Gal(K/\Q)}$. Recall that the embedding  from Theorem \ref{Main-Theorem} is given by the map 
	\begin{align*}
		\varphi \colon &\Gal(L/\Q) \longrightarrow \Gal(L/K) \wr_r \Gal(K/\Q) \\
		&\quad \quad \qquad  \rho \longmapsto (\sigma_{\rho}, \varepsilon(\rho))
	\end{align*}
	where $\sigma_{\rho} \colon \Gal(K/\Q) \longrightarrow \Gal(L/K)$ is defined by 
	$$
	\sigma_{\rho}(\tau) := s(\tau)^{-1} \circ \rho \circ s \big( \varepsilon(\rho)^{-1} \circ \tau \big)
	$$ 
	for every $\tau \in \Gal(K/\Q)$. We now compute explicitly the maps $\sigma_{\rho}$ for every $\rho \in \Gal(L/\Q)$. For example, for $\sigma_{\op{id}_L}$ we have
	
	\begin{align*}
		\sigma_{\op{id}_L}(\op{id}_K) = s(\op{id}_K)^{-1} \circ \op{id_{L}} \circ s \big( \varepsilon(\op{id}_L)^{-1} \circ \op{id}_K  \big)	= \op{id}_L^{-1} \circ s \big( \op{id}_K^{-1} \big) = \op{id}_L \circ \op{id}_L = \op{id}_L
 	\end{align*}
 	and similarly
 	\begin{align*}
		\sigma_{\op{id}_L}(\eta) = s(\eta)^{-1} \circ \op{id_{L}} \circ s \big( \varepsilon(\op{id}_L)^{-1} \circ \eta  \big)	= \rho_1^{-1} \circ s \big( \op{id}_K^{-1} \circ \eta \big) = \rho_1 \circ s(\eta) = \rho_1 \circ \rho_1 = \rho_1^2 = \op{id}_L.
 	\end{align*}
 	This shows that $\sigma_{\op{id}_L} = f_1$. Analogous calculations show that $\sigma_{\rho_1} = f_1$, $\sigma_{\rho_2} = f_4$ and $\sigma_{\rho_3} = f_4$. Therefore, the embedding $\varphi \colon \Gal(L/\Q) \longrightarrow \Gal(L/K) \wr_r \Gal(K/\Q)$ is given by
 	 	
 	\begin{align*}
 		\varphi(\op{id}_L) &= (\sigma_{\op{id}_L}, \varepsilon(\op{id}_L)) = (f_1, \op{id}_K),\\
 		\varphi(\rho_1) &= (\sigma_{\rho_1}, \varepsilon(\rho_1)) = (f_1, \eta),\\
 		\varphi(\rho_2) &= (\sigma_{\rho_2}, \varepsilon(\rho_2)) = (f_4, \op{id}_K),\\
 		\varphi(\rho_3) &= (\sigma_{\rho_3}, \varepsilon(\rho_3)) = (f_4, \eta).
 	\end{align*}
	
	In order to look at the group structure of the wreath product we can use the isomorphisms $\Gal(L/K) \simeq C_2$ and $\Gal(K/\Q) \simeq C_2$, as well as Corollary \ref{Regular-wreath-isomorphism} to conclude that $$\Gal(L/K) \wr_r \Gal(K/\Q) \simeq C_2 \wr_r C_2.$$ Then, from the calculations carried out in Example \ref{example-c2-regular-c2} we conclude that this wreath product is isomorphic to $D_4$. However, we can determine this structure manually as well. First we recall from Proposition \ref{theta-homomorphism} that the homomorphism $\theta$ is given by
	
	\begin{align*}
	\theta \colon &\Gal(K/\Q) \longrightarrow \op{Aut}(\Gal(L/K)^{\Gal(K/\Q)})\\
	& \qquad \qquad h \longmapsto \theta_{h} 
\end{align*}
where $\theta_{h} \colon \Gal(L/K)^{\Gal(K/\Q)} \longrightarrow \Gal(L/K)^{\Gal(K/\Q)}$ is the automorphism defined by $$\theta_h(f)(\tau) = f(h^{-1} \circ \tau)$$ for every $f \in \Gal(L/K)^{\Gal(K/\Q)}$ and every $\tau \in \Gal(K/\Q)$. From this we can compute $\theta_{\op{id}_K}(f)(\tau) = f(\op{id}_K^{-1} \circ \tau) = f(\tau)$ for every $f \in \Gal(L/K)^{\Gal(K/\Q)}$, so that $\theta_{\op{id}_K}$ is the identity automorphism in $\op{Aut}(\Gal(L/K)^{\Gal(K/\Q)})$. For the other automorphism we have 
$$
\theta_{\eta}(f)(\tau) = f(\eta^{-1} \circ \tau) = f(\eta \circ \tau).
$$
Therefore, we see that 
\begin{align*}
	&\theta_{\eta}(f_1) \colon \op{id}_K \mapsto f_1(\eta \circ \op{id}_K) = f_1(\eta) = \op{id}_L \qquad \eta \mapsto f_1(\eta \circ \eta ) = f_1(\op{id}_K) = \op{id}_L,\\
	&\theta_{\eta}(f_2) \colon \op{id}_K \mapsto f_2(\eta \circ \op{id}_K) = f_2(\eta) = \rho_1 \qquad \eta \mapsto f_2(\eta \circ \eta ) = f_2(\op{id}_K) = \op{id}_L,\\
	&\theta_{\eta}(f_3) \colon \op{id}_K \mapsto f_3(\eta \circ \op{id}_K) = f_3(\eta) = \op{id}_L \qquad \eta \mapsto f_3(\eta \circ \eta ) = f_3(\op{id}_K) = \rho_1,\\
	&\theta_{\eta}(f_4) \colon \op{id}_K \mapsto f_4(\eta \circ \op{id}_K) = f_4(\eta) = \rho_1 \qquad \eta \mapsto f_4(\eta \circ \eta ) = f_4(\op{id}_K) = \rho_1.\\
\end{align*}
This shows that 
$$
\theta_{\eta} \colon f_1 \mapsto f_1, \quad f_2 \mapsto f_3, \quad f_3 \mapsto f_2, \quad f_4 \mapsto f_4.
$$
Thus $\theta_{\eta}$ just permutes $f_2$ and $f_3$ and leaves $f_1$ and $f_4$ fixed. Using this fact, if we define the elements $r:=(f_2,\eta)$ and $s:=(f_1,\eta)$ of $\Gal(L/K) \wr_r \Gal(K/\Q)$, it is a short calculation to verify that $r$ and $s$ are generators of this group, and moreover that $\text{ord}(r)=4, \text{ord}(s)=2$ and that $sr=r^{-1}s$, thus proving that the wreath product is isomorphic to the dihedral group $D_4$, as was already mentioned above as a consequence of Corollary \ref{Regular-wreath-isomorphism} and Example \ref{example-c2-regular-c2}.
\end{example}

\begin{example}[A Kummer extension]\label{Kummer-example}
Let $K := \Q(\alpha)$, where $\alpha$ is a root of the irreducible polynomial $p(x):= x^6 - x^5 - x^2 - x + 1 \in \Q[x]$. Then $K/\Q$ is a degree 6 extension with discriminant $d_K = -11691 = -3^3 \cdot 433$. Thus, the only rational primes that ramify in $K$ are 3 and 433. This number field corresponds to the entry \cite[\href{https://www.lmfdb.org/NumberField/6.0.11691.1}{Number field 6.0.11691.1}]{LMFDB} in the \texttt{LMFDB} database. It turns out that the only intermediate field of the extension $\Q \subseteq K$ is the quadratic field $\Q(\sqrt{-3})$. In particular, this implies that $K$ contains the 6-th roots of unity $\{ \pm 1, \frac{1}{2}(\pm 1 \pm \sqrt{-3})  \}$. Moreover, the Galois closure $K^c$ of $K/\Q$ is a number field of degree $[K^c : \Q] = 72$ and $\Gal(K^c/\Q) \simeq S_3 \wr_{\Omega_K} C_2 = S_3^2 \rtimes C_2$, where $\Omega_K = \{ 1, 2 \}$ and $C_2$ acts on $\Omega_K$ by multiplication and reduction modulo $2$ by using the elements $1$ and $2$ as representatives of the residue classes.

Now, let $L := K(\sqrt{2}, \sqrt[3]{5})$. Then by Theorem \ref{Kummer-extensions} we know that $L/K$ is a 6-Kummer extension. This means that $L/K$ is an abelian extension with abelian Galois group of exponent 6. For example, in \texttt{SageMath} \cite{SageMath} the number field $L$ can be defined with the following commands.
\begin{lstlisting}[language=Python]
#We define the number field L
M = NumberField([x^6 - x^5 - x^2 - x + 1, x^2 - 2, x^3 - 5], 'a, b, c')
L.<u> = M.absolute_field()
\end{lstlisting}
This number field has discriminant $d_L = 2^{54} \cdot 3^{36} \cdot 5^{24} \cdot 359^{6} \cdot 443^{6}$ and moreover $[L:\Q] = 36$. Therefore, by the tower law we know that $[L : K] = 6$. Since the extension $L/K$ is abelian (and thus Galois) of degree 6, and since there is only one abelian group of order 6, we know that $\Gal(L/K) \simeq C_6$. 

Then, we tried computing directly the Galois closure $L^c$ of $L/\Q$ in \texttt{SageMath} with the command \texttt{galois\_closure()}, but the computation was very heavy and it ended up exceeding the default stack space used in the algorithm. Also, the number field $L$ still does not appear in the \texttt{LMFDB} database, as a search for number fields of degree 36 with discriminant $d_L$ turned up empty. Thus, we instead compute the degree $[L^c:\Q]$ by using Lemma \ref{Galois-closure-2} as follows. A direct application of that lemma implies that 
$$
L^c = K^c(\sqrt{2}, \sqrt[3]{5}).
$$
Now, as was mentioned above, the rational primes that ramify in $K$ are 3 and 433 and it is known that a prime ramifies in $K$ if and only if it ramifies in the Galois closure $K^c$; this follows from the fact that the Galois closure $K^c$ is the compositum of all the Galois conjugate fields of $K$ and that if a prime $p \in \Z$ is unramified in the number fields $K_1$ and $K_2$, it is also unramified in their compositum $K_1K_2$ (see e.g. \cite[p. 159, Corollary 1]{Nar04}). Therefore the only ramified primes in $K^c$ are 3 and 433. This implies that $\sqrt{2} \notin K^c$ and $\sqrt[3]{5} \notin K^c$ because otherwise 2 or 5 would be ramified in $K^c$. Therefore, this implies that 
$$
[L^c : \Q] = [K^c(\sqrt{2}, \sqrt[3]{5}):\Q] = [K^c(\sqrt{2}, \sqrt[3]{5}):K^c] [K^c:\Q] = 6 \cdot 72 = 432.
$$
Then, by Theorem \ref{Kummer-embedding} we know that $\Gal(L^c/\Q)$, which is a group of order $432$, can be embedded into the wreath product 
\begin{align}\label{Kummer-example-wreath}
\Gal(L/K) \wr_{\Omega} \Gal(K^c/\Q) \simeq C_6 \wr_{\widehat{\Omega}} (S_3 \wr_{\Omega_K} C_2),  
\end{align}
where $\Omega := \operatorname{Hom}_{\Q}(K, K^c)$ is the set of embeddings, which has 6 elements, and similarly $\widehat{\Omega} = \{ 1, 2, 3, 4, 5, 6 \}$.
Note that this wreath product has order 
$$
|\Gal(L/K) \wr_{\Omega} \Gal(K^c/\Q)| = |\Gal(L/K)|^{|\Omega|} \cdot |\Gal(K^c/\Q)| = 6^{6} \cdot 72 = \num{3359232} = 2^9 \cdot 3^8.
$$ 

Finally, observe that Theorem \ref{Main-Theorem-2} gives an embedding of $\Gal(L^c/\Q)$ into the wreath product 
$$
\Gal(L^c/K) \wr_{\Omega} \Gal(K^c/\Q),
$$
which has order $|\Gal(L^c/K) \wr_{\Omega} \Gal(K^c/\Q)| = |\Gal(L^c/K)|^{|\Omega|} \cdot |\Gal(K^c/\Q)| = \left(  \frac{432}{6}  \right)^6 \cdot 72 = \num{10030613004288} = 2^{21} \cdot 3^{14}$. To compare, this is a group that is about 40 million times bigger than the wreath product \eqref{Kummer-example-wreath}! 

Thus, this example shows, in a very striking fashion, how the embedding from Theorem \ref{Kummer-embedding} is significantly sharper than the one from Theorem \ref{Main-Theorem-2}.
\end{example}



\bibliography{wreath.bib}
\bibliographystyle{alphaurl}

\end{document}